\documentclass[a4paper, 11pt]{amsart}
\title[{\tiny stability conditions on CY double/triple solids}]
{stability conditions on Calabi-Yau double/triple solids}
\date{}
\author{Naoki Koseki}

\makeatletter
 
  \@addtoreset{equation}{section}
 \makeatother

\usepackage{amsmath, amssymb, amsthm, amscd, comment, mathtools, color}
\usepackage[frame,cmtip,curve,arrow,matrix,line,graph]{xy}
\usepackage{tikz}
\usetikzlibrary{intersections, calc}

\theoremstyle{plain}
\newtheorem{thm}{Theorem}[section]
\newtheorem{prop}[thm]{Proposition}
\newtheorem{lem}[thm]{Lemma}
\newtheorem{cor}[thm]{Corollary}
\newtheorem*{thm*}{Theorem}

\theoremstyle{definition}
\newtheorem{defin}[thm]{Definition}

\newtheorem*{NaC}{Notation and Convention}
\newtheorem*{ACK}{Acknowledgement}

\theoremstyle{remark}
\newtheorem{rmk}[thm]{Remark}

\newtheorem{ques}[thm]{Question}

\newtheorem{assump}[thm]{Assumption}

\DeclareMathOperator{\ch}{ch}
\DeclareMathOperator{\td}{td}

\newcommand{\bP}{\mathbb{P}}
\newcommand{\bC}{\mathbb{C}}
\newcommand{\bR}{\mathbb{R}}
\newcommand{\bZ}{\mathbb{Z}}
\newcommand{\mcA}{\mathcal{A}}

\newcommand{\mcO}{\mathcal{O}}

\newcommand{\mcF}{\mathcal{F}}
\newcommand{\mcH}{\mathcal{H}}

\newcommand{\mcT}{\mathcal{T}}

\DeclareMathOperator{\Hom}{Hom}

\DeclareMathOperator{\Coh}{Coh}
\DeclareMathOperator{\Supp}{Supp}

\DeclareMathOperator{\ext}{ext}
\DeclareMathOperator{\Ext}{Ext}

\DeclareMathOperator{\Stab}{Stab}
\DeclareMathOperator{\cl}{cl}

\DeclareMathOperator{\GL}{GL}

\DeclareMathOperator{\Cone}{Cone}

\DeclareMathOperator{\lcm}{lcm}

\begin{document}
\maketitle

\begin{abstract}
In this paper, we prove a stronger form 
of the Bogomolov-Gieseker (BG) inequality 
for stable sheaves 
on two classes of Calabi-Yau threefolds, 
namely, weighted hypersurfaces inside 
the weighted projective spaces 
$\bP(1, 1, 1, 1, 2)$ and $\bP(1, 1, 1, 1, 4)$. 

Using the stronger BG inequality 
as a main technical tool, 
we construct open subsets in the spaces of 
Bridgeland stability conditions on these Calabi-Yau threefolds. 
\end{abstract}

\setcounter{tocdepth}{1}
\tableofcontents

\section{Introduction} \label{section:introduction}
\subsection{Motivation and Results}
Since Bridgeland \cite{bri07} defined the notion of 
stability conditions on derived categories, 
its construction on a given threefold has been 
an important open problem. 
It turned out that, to solve this problem, 
we need a Bogomolov-Gieseker (BG) type inequality, 
involving the third Chern character, 
for certain stable objects in the derived category 
\cite{bms16, bmt14a, bmsz17}. 
There are several classes of threefolds on which 
we know the existence of Bridgeland stability conditions 
\cite{bms16, bmt14a, bmsz17, kos17, kos20, li19b, li19a, 
liu21a, liu21b, mp16a, mp16b, mac14, piy17, sch14, sun19}. 
For $K$-trivial threefolds, the only known cases are 
the quintic threefolds \cite{li19b}, 
Abelian threefolds, 
and their {\' e}tale quotients \cite{bms16, mp16a, mp16b}. 

Among them, Li \cite{li19b} recently treated quintic threefolds, 
which is one of the most important cases 
for Mirror Symmetry. 
The crucial step in his arguments is 
to establish the improvement of 
the classical BG inequality 
for torsion free slope stable sheaves. 
Recall that a version of 
the classical BG inequality 
is the inequality 
\begin{equation} \label{eq:BGintro}
\frac{H\ch_2(E)}{H^3\ch_0(E)} 
\leq \frac{1}{2}\left( \frac{H^2\ch_1(E)}{H^3\ch_0(E)} \right)^2, 
\end{equation}
where $E$ is a slope stable sheaf 
with respect to an ample divisor $H$. 
For del Pezzo and K3 surfaces, 
we can easily get the inequality 
stronger than (\ref{eq:BGintro}), 
simply by using the Serre duality. 
In contrast, such an improvement 
of the BG inequality on Calabi-Yau threefolds 
is highly non-trivial. 

When the first draft of this paper was submitted, 
the arguments in \cite{li19b} 
have been applied only for quintic threefolds. 
Very recently, Liu \cite{liu21a} treated Calabi-Yau complete intersections 
of quadratic and quartic hypersurfaces in $\bP^5$ via a similar method. 

The goal of the present paper is 
to extend it to two other examples of 
Calabi-Yau threefolds, 
namely, general weighted hypersurfaces 
in the weighted projective spaces 
$\bP(1, 1, 1, 1, 2)$ and 
$\bP(1, 1, 1, 1, 4)$. 
We call them as 
{\it triple/double cover CY3}, 
since they have finite morphisms 
to $\bP^3$ of degree $3$, $2$, respectively. 
The following is our main result: 

\begin{thm}[Theorems \ref{thm:strongBG}, \ref{thm:BGonT2}] \label{thm:mainintro}
Let $X$ be a double or triple cover CY3, 
$H$ the primitive ample divisor, 
and $E$ a slope stable sheaf 
with slope $\mu \in [-1, 1]$. 
Then the inequality 
\begin{equation} \label{eq:strongBGintro}
\frac{H\ch_2(E)}{H^3\ch_0(E)}
\leq \Xi\left(\left|
	\frac{H^2\ch_1(E)}{H^3\ch_0(E)}
	\right|\right)
\end{equation}
holds. Here the function 
$\Xi$ is defined as follows. 
\[
\Xi(t):=\left\{ \begin{array}{ll}
t^2-t & (t \in [0, 1/4]) \\ 
3t/4-3/8 & (t \in [1/4, 1/2]) \\
t/4-1/8 & (t \in [1/2, 3/4]) \\
t^2-1/2 & (t \in [3/4, 1]). 
\end{array} \right.
\]
\end{thm}

\begin{figure}[htb]
\begin{center}
\begin{tikzpicture}[scale=4]
\draw[->] (-1, 0) -- (1, 0) node[right]{$\frac{H^2\ch_1}{H^3\ch_0}$}; 
\draw[->] (0, -1/2) -- (0, 1) node[right]{$\frac{H\ch_2}{H^3\ch_0}$}; 
\draw[red, thick, domain=-1:-3/4] plot(\x, \x*\x-1/2);
\draw[red, thick, domain=-3/4:-1/2] plot(\x, -1/4*\x-1/8);
\draw[red, thick, domain=-1/2:-1/4] plot(\x, -3/4*\x-3/8); 
\draw[red, thick, domain=-1/4:0] plot(\x, \x*\x+\x); 
\draw[red, thick, domain=0:1/4] plot(\x, \x*\x-\x); 
\draw[red, thick, domain=1/4:1/2] plot(\x, 3/4*\x-3/8); 
\draw[red, thick, domain=1/2:3/4] plot(\x, 1/4*\x-1/8); 
\draw[red, thick, domain=3/4:1] plot(\x, \x*\x-1/2); 
\draw[domain=-1:1] plot(\x, 1/2*\x*\x); 
\end{tikzpicture}
\end{center}
\caption{strong BG inequality on double/triple cover CY3s.} \label{fig:BGonXintro}
\end{figure}
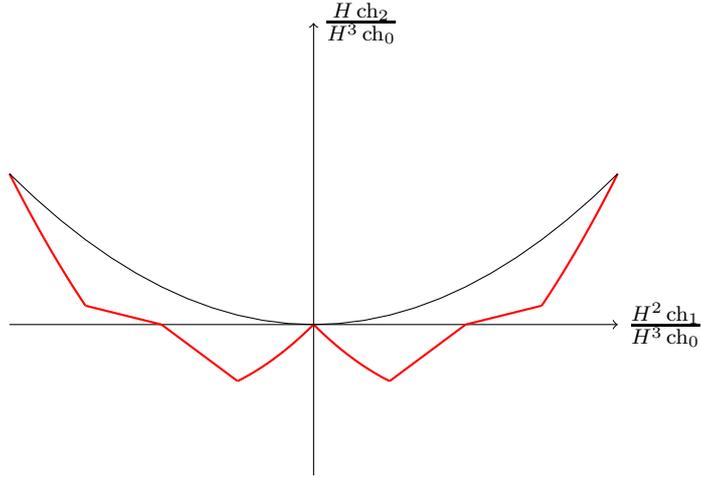

Using this stronger BG inequality, 
we prove the following BG type inequality 
(involving $\ch_3$) 
for $\nu_{\beta, \alpha}$-stable objects, 
which are certain two term complexes 
in the derived category. 
For the precise definition of 
$\nu_{\beta, \alpha}$-stability, 
see Section \ref{section:preliminaries}. 

\begin{thm}[Theorem \ref{thm:reduction}, Corollary \ref{cor:bgconj}] \label{thm:main2intro}
Let $X$ be a double or triple cover CY3, 
take real numbers $\alpha, \beta \in \bR$ with 
$\alpha > 
\frac{1}{2}\beta^2
+\frac{1}{2}(\beta-\lfloor \beta \rfloor)
(\lfloor \beta \rfloor+1-\beta)$. 
Let $E$ be a $\nu_{\beta, \alpha}$-semistable object. 
Then the inequality 
\[
Q^\Gamma_{\alpha, \beta}(E) \geq 0 
\]
holds. 
Here, we put $\Gamma:=\frac{2}{9}H^2$ (resp. $\frac{1}{3}H^2$) 
when $X$ is a triple (resp. double) cover CY3, 
and the quadratic form 
$Q^\Gamma_{\alpha, \beta}$ 
is defined as follows: 
\begin{align*}
Q^\Gamma_{\alpha, \beta}(E)&:=
(2\alpha-\beta^2)\left(
\overline{\Delta}_{H}(E)+3\frac{\Gamma.H}{H^3}
	\left(H^3\ch^\beta_0(E)
	\right)^2
	\right) \\
&\quad +2\left(H\ch^\beta_2(E) \right)\left(
	2H\ch^\beta_2(E)-3\Gamma.H \ch^\beta_0(E)
	\right) \\
&\quad -6\left(H^2\ch^\beta_1(E) \right) \left(
	\ch^\beta_3(E)-\Gamma\ch_1^\beta(E) 
	\right). 
\end{align*}

\end{thm}

The above theorem enables us 
to construct an open subset 
in the space of Bridgeland stability conditions 
\cite{bms16, bmt14a, bmsz17}. 
For real numbers $\alpha, \beta, a, b$, 
we define a group homomorphism 
$Z^{a, b}_{\beta, \alpha} \colon K(X) \to \bC$ as 
\[
Z^{a, b}_{\beta, \alpha}:=
-\ch_3^{\beta}+bH\ch_2^\beta+aH^2\ch_1^\beta
+i\left(
	H\ch_2^\beta-\frac{1}{2}\alpha^2H^3\ch_0^\beta 
	\right). 
\]
We denote by $\mcA^{\beta, \alpha}$ 
the double-tilted heart 
defined in \cite{bmt14a}. 

\begin{thm}[Theorem \ref{thm:familystab}] \label{thm:main3intro}
We have a continuous family 
$\left(Z^{a, b}_{\beta, \alpha}, \mcA^{\beta, \alpha} \right)$ 
of stability conditions 
parametrized by real numbers 
$\alpha, \beta, a, b$ satisfying 
\[
\alpha> 0, \quad 
\alpha^2+\left(\beta-\lfloor\beta \rfloor-\frac{1}{2} \right)^2 > \frac{1}{4}, 
\quad a > \frac{1}{6}\alpha^2+\frac{1}{2}|b|\alpha+\gamma, 
\]
where we put $\gamma:=2/9$ (resp. $1/3$) 
when $X$ is a triple (resp. double) cover CY3. 
Acting by the group $\widetilde{\GL^+}(2; \bR)$, 
it forms an open subset in the space of stability conditions. 
\end{thm}

\subsection{Strategy of the proof}
In this subsection, we briefly explain 
how to prove Theorem \ref{thm:mainintro}. 
Let us first recall the arguments in \cite{li19b} 
for a quintic threefold $X_5 \subset \bP^4$. 
We consider $(2, 2, 5)$, $(2, 5)$, $(2, 2)$ 
complete intersections 
\[
C_{2, 2, 5} \subset T_{2, 5} \subset X_5, 
\quad C_{2, 2, 5} \subset S_{2, 2}. 
\]
The stronger BG inequality on $X_5$ 
is proved in the following way: 
\begin{enumerate}
\item First we reduce the problem 
to proving the same inequality for stable sheaves 
on the surface $T_{2, 5} \subset X_5$, 
by using the restriction technique. 

\item Again using the restriction, 
the problem is further reduced to 
establishing a stronger Clifford type bounds 
on global sections for stable vector bundles 
on the curve $C_{2, 2, 5} \subset T_{2, 5}$. 

\item Regard the stable vector bundle on $C_{2, 2, 5}$ 
as a torsion sheaf on the surface $S_{2, 2}$ 
via the inclusion $C_{2, 2, 5} \subset S_{2, 2}$. 
Then a wall-crossing argument 
in the space of Bridgeland stability conditions 
on the surface $S_{2, 2}$ 
gives the desired Clifford type bounds. 
The argument in this step 
first appeared in \cite{fey19}.  
\end{enumerate}
In step (3), the crucial fact is 
that the surface $S_{2, 2}$ is del Pezzo, 
on which a stronger BG inequality holds. 

For double/triple cover CY3s, 
the situation is quite similar. 
In fact, we have smooth complete intersection varieties 
\begin{align*}
&C_{2, 2, 6} \subset T_{2, 6} \subset X_6, \quad 
C_{2, 2, 6} \subset S_{2, 2} \quad \mbox{ in } 
\bP(1, 1, 1, 1, 2), \\
&C_{2, 4, 8} \subset T_{2, 8} \subset X_8, \quad 
C_{2, 4, 8} \subset S_{2, 4} \quad \mbox{ in }
\bP(1, 1, 1, 1, 4), 
\end{align*}
where both of the surfaces 
$S_{2, 2} \subset \bP(1, 1, 1, 1, 2)$ and 
$S_{2, 4} \subset \bP(1, 1, 1, 1, 4)$ 
are isomorphic to the quadric surface $\bP^1 \times \bP^1$, 
which is del Pezzo. 
Note that we consider $(2, 4)$ complete intersection 
in $\bP(1, 1, 1, 1, 4)$ instead of $(2, 2)$, 
to avoid the singularity. 

Hence we are able to 
apply the methods in \cite{li19b} 
to our cases. 
At this moment, we do not know the way 
to treat these examples uniformly, 
so the author believes 
it is still worth writing down the complete proofs. 
In fact, it turns out that, in our cases, 
we need the modified term $\Gamma$ 
in Theorem \ref{thm:main2intro}, 
unlike the quintic case. 

\subsection{Open problems}
\begin{enumerate}
\item In Theorem \ref{thm:main2intro} 
we expect we can take $\Gamma=0$. 
For this, we need a further improvement 
of Theorem \ref{thm:mainintro}. 

\item Stability conditions we construct in this paper 
are said to be `near the large volume limit' in Physics. 
For weighted hypersurfaces, 
we expect the existence of 
another kind of stability conditions, 
called {\it Gepner type}. 
Mathematically, it is the stability condition 
invariant under the certain autoequivalence 
of the derived category. 
See \cite{tod14b, tod17} for discussions 
on the construction of Gepner type stability conditions. 
To construct the heart corresponding to 
the Gepner type stability condition, 
the first task is to prove a stronger form 
of the BG inequality for stable sheaves 
with a specific slope equal $-1/2$. 
Unfortunately, Theorem \ref{thm:mainintro} 
is not enough for this purpose. 

\item One might ask whether we can treat 
other Calabi-Yau weighted hypersurfaces inside 
$\bP(a_1, a_2, a_3, a_4, a_5)$ 
with more general weights $(a_i)$. 
Unfortunately, quintic and double/triple cover CY3s 
are the only cases where 
$\bP(a_1, a_2, a_3, a_4, a_5)$ 
contains a smooth Calabi-Yau hypersurface 
and a smooth del Pezzo (or K3) complete intersection surface 
at the same time. 
Indeed, it happens precisely when 
the weighted $\bP^4$ has only isolated singularities 
and its canonical line bundle can be written 
as $L^{\otimes m}$, 
where $L$ is a free line bundle 
and $m \geq 2$. 
These conditions are equivalent to 
the following numerical conditions. 
\begin{itemize}
\item for any $i$ with $a_i>1$ and for any $j \neq i$, 
$a_i$ does not divide $a_j$, 
\item $\sum a_i=m \cdot \lcm(a_i)$. 
\end{itemize}

An easy but lengthy calculation show that 
there are only three solutions. 
If we allow smooth Deligne-Mumford stacks, 
i.e., if we allow 
the weighted $\bP^4$ to have 
non-isolated singularities, 
there are several other solutions. 
\end{enumerate}

\subsection{Plan of the paper}
The paper is organized as follows. 
In Section \ref{section:preliminaries}, 
we recall about the notion of tilt stability 
in the derived category, 
and about the BG type inequality conjecture. 
Sections \ref{section:clifford} and \ref{section:strongBG} 
are devoted to proving 
Theorem \ref{thm:mainintro} 
for a triple cover CY3. 
The key ingredient is the stronger Clifford type bound 
proved in Section \ref{section:clifford}. 
In Section \ref{section:2to1}, we treat 
the case of a double cover CY3. 
In Section \ref{section:BGconj}, 
we prove Theorem \ref{thm:main2intro}. 
Finally, in Section \ref{section:conststab}, 
we prove Theorem \ref{thm:main3intro}.

\begin{ACK}
The author would like to thank Professor Arend Bayer 
for giving him various advices and comments. 
The author would also like to thank 
Chunyi Li, Masaru Nagaoka, Genki Ouchi 
and Professor Yukinobu Toda 
for related discussions. 
The author was supported by 
ERC Consolidator grant WallCrossAG, no.819864. 

Finally, the author would like to thank the referee for careful reading of the previous version of this article and giving him various useful comments. 
\end{ACK}

\begin{NaC}
In this paper, we always work over the complex number field $\bC$. 
We will use the following notations. 
\begin{itemize}
\item For an ample divisor $H$ and a real number $\beta \in \bR$, 
we denote by 
$\ch^\beta=(\ch^\beta_0, \cdots, \ch^\beta_n):=e^{-\beta H}\ch$, 
the $\beta$-twisted Chern character.  

\item $\hom(E, F):=\dim \Hom(E, F)$, and 
$\ext^i(E, F):=\dim\Ext^i(E, F)$ 
for objects $E, F$ in the derived category, and an integer $i$. 
\end{itemize}
\end{NaC}

\section{Preliminaries} \label{section:preliminaries}
\subsection{BG type inequality conjecture}
In this subsection, we recall the notion of tilt stability, 
and the BG type inequality conjecture. 
We mainly follow the notations in the paper \cite{li19b}. 
Let $X$ be a smooth projective variety 
of dimension $n \geq 2$, 
$H$ an ample divisor. 
We take real numbers $\alpha, \beta \in \bR$ 
with $\alpha > \frac{1}{2} \beta^2$. 
We define a slope function $\mu_H$ as follows: 
\[
\mu_H:=\frac{H^{n-1}\ch_1}{H^n\ch_0} \colon 
\Coh(X) \to \bR \cup \{+\infty\}. 
\]
We have the notion of $\mu_H$-stability on $\Coh(X)$, 
and the corresponding torsion pair on $\Coh(X)$ : 
\begin{align*}
&\mcT_\beta:=\left\langle
T \in \Coh(X) : 
T \mbox{ is } \mu_H \mbox{-semistable with } \mu_H(T) >\beta
\right\rangle, \\
&\mcF_\beta:=\left\langle
F \in \Coh(X) : 
F \mbox{ is } \mu_H \mbox{-semistable with } \mu_H(F) \leq \beta
\right\rangle. 
\end{align*}

Here, $\langle S \rangle$ 
denotes the extension closure 
of a set $S \subset \Coh(X)$ 
of objects in the category $\Coh(X)$. 
By the general theory of torsion pairs \cite{hrs96}, 
we obtain the new abelian category 
\[
\Coh^\beta(X):=\left\langle
\mcF_\beta[1], \mcT_\beta
\right\rangle 
\subset D^b(X), 
\]
which is the heart of a bounded t-structure on $D^b(X)$. 
On the heart $\Coh^\beta(X)$, 
we define the following slope function : 
\[
\nu_{\beta, \alpha}:=
\frac{H^{n-2}\ch_2-\alpha H^n\ch_0}{H^{n-1}\ch_1-\beta H^n\ch_0} 
\colon \Coh^\beta(X) \to \bR \cup \{+\infty\}. 
\]
Then as similar to the $\mu_H$-stability on $\Coh(X)$, 
we can define the notion of $\nu_{\beta, \alpha}$-stability on $\Coh^\beta(X)$. 
We also call $\nu_{\alpha, \beta}$-stability as {\it tilt-stability}. 

\begin{defin} \label{def:bn}
Let $E \in \Coh^0(X)$ be an object.  
\begin{enumerate}
\item We define the {\it Brill-Noether (BN)} slope of $E$ as 
\[
\nu_{BN}(E):=
\frac{H^{n-2}\ch_2(E)}{H^{n-1}\ch_1(E)} 
\in \bR \cup \{+\infty\}. 
\]

\item We say the object $E$ is 
{\it Brill-Noether (BN) (semi)stable} 
if it is $\nu_{0, \alpha}$-(semi)stable 
for every sufficiently small real number 
$0 < \alpha \ll 1$. 
\end{enumerate}
\end{defin}

We refer \cite[Section 2]{li19b} for the basic properties 
of tilt stability and BN stability. 
Let us define the discriminant 
of an object $E \in D^b(X)$ as 
\[
\overline{\Delta}_H(E):=
(H^{n-1}\ch_1(E))^2-2H^n\ch_0(E)H^{n-2}\ch_2(E). 
\]
The following is the main question we investigate in this paper. 

\begin{ques}[\cite{bms16, bmt14a, bmsz17}] \label{conj:BG}
Assume that $n=\dim X=3$. 
Find a $1$-cycle $\Gamma \in A_1(X)_{\bR}$ 
satisfying $\Gamma.H \geq 0$, and the following property: 
Let $E$ be a $\nu_{\beta, \alpha}$-semistable object. 
Then the inequality 
\[
Q^\Gamma_{\alpha, \beta}(E) \geq 0 
\]
holds. Here, the quadratic form $Q^\Gamma_{\alpha, \beta}$ is defined as follows: 
\begin{align*}
Q^\Gamma_{\alpha, \beta}(E)&:=
(2\alpha-\beta^2)\left(
\overline{\Delta}_{H}(E)+3\frac{\Gamma.H}{H^3}
	\left(H^3\ch^\beta_0(E)
	\right)^2
	\right) \\
&\quad +2\left(H\ch^\beta_2(E) \right)\left(
	2H\ch^\beta_2(E)-3\Gamma.H \ch^\beta_0(E)
	\right) \\
&\quad -6\left(H^2\ch^\beta_1(E) \right) \left(
	\ch^\beta_3(E)-\Gamma\ch_1^\beta(E) 
	\right). 
\end{align*}
\end{ques}

The conjectural inequality above 
is called {\it the Bogomolov-Gieseker(BG) type inequality conjecture}, 
proposed in \cite{bms16, bmt14a} with $\Gamma=0$. 
It is known that the BG type inequality conjecture with $\Gamma=0$ 
fails for some classes of threefolds, such as 
the blow-up of $\bP^3$ at a point 
(cf. \cite{kos17, ms19, sch17}). 
The question with the modified term $\Gamma$ 
appeared in \cite{bmsz17} 
and proved affirmatively for all Fano threefolds. 

The following reduction of Question \ref{conj:BG} 
plays an important role in this paper. 
\begin{thm}[cf. {\cite[Theorem 3.2]{li19b}}] \label{thm:reduction}
Assume that $n=\dim X=3$. 
Let $\Gamma$ be a $1$-cycle with $\Gamma.H \geq 0$. 
Suppose that for every BN stable object with 
$\nu_{BN}(E) \in [0, 1/2]$, the inequality 
$Q^\Gamma_{0, 0}(E) \geq 0$ holds. 

Then the inequality in Question \ref{conj:BG} holds 
for any choice of real numbers 
$\alpha, \beta \in \bR$ with 
$\alpha > 
\frac{1}{2}\beta^2
+\frac{1}{2}(\beta-\lfloor \beta \rfloor)
(\lfloor \beta \rfloor+1-\beta)$. 
\end{thm}
\begin{proof}
Exactly the same arguments 
as in \cite[Theorem 3.2]{li19b} work 
since the following statements are true. 
\begin{itemize}
\item Let $(\beta', \alpha') \in \bR^2$ be a point 
on the line through $p_H(E)$ and $(\beta, \alpha)$ 
with $\alpha'>\frac{1}{2}\beta'^2$. 
Then $Q^{\Gamma}_{\alpha, \beta}(E) <0$ implies
 $Q^{\Gamma}_{\alpha', \beta'}(E) <0$. 
Here we define a point $p_H(E) \in \bR^2$ as 
\[
p_H(E):=\left(
	\frac{H^2\ch_1(E)}{H^3\ch_0(E)}, 
	\frac{H\ch_2(E)}{H^3\ch_0(E)}
	\right). 
\]
 
 \item The quadratic form 
 $Q^{\Gamma}_{\alpha, \beta}$ 
 is semi-negative definite on the kernel of 
 $\overline{Z}_{\alpha, \beta}:=H^2\ch_1^\beta+i(H\ch_2-\alpha H^3\ch_0)$. 
\end{itemize}
\end{proof}

\subsection{Star-shaped functions and the BG type inequalities}
In this subsection, we explain 
the wall-crossing technique 
used to obtain the (stronger) BG inequality 
for tilt-stable objects. 
This idea will also appear in the proof of 
the BG type inequality conjecture involving $\ch_3$. 
As in the previous subsection, 
we denote by $X$ a smooth projective variety of dimension $n$, 
and $H$ an ample divisor on $X$. 
We use the following notion. 
\begin{defin}
A function $f \colon \bR \to \bR$ 
is called {\it star-shaped} 
if the following condition hold: 
For all real numbers 
$\alpha, \beta \in \bR$ with $\alpha >0$, 
the line segment connecting the points 
$(\beta, f(\beta))$ and $(0, \alpha)$ is above 
the graph of $f$. 
\end{defin}

Recall that for an object $E \in D^b(X)$ 
with $\ch_0(E) \neq 0$, 
we define 
\[
p_H(E):=\left(
	\frac{H^{n-1}\ch_1(E)}{H^n\ch_0(E)}, 
	\frac{H^{n-2}\ch_2(E)}{H^n\ch_0(E)}
	\right). 
\]
We have the following result: 
\begin{prop}[cf. \cite{bms16, li19a}] 
\label{prop:bg-tilt*}
Let $f \colon \bR \to \bR$ be a star-shaped function. 
Assume that for every $\mu_H$-semistable torsion free sheaf $E$, 
the inequality 
\[
\frac{H^{n-2}\ch_2(E)}{H^n\ch_0(E)} \leq 
f\left( 
\frac{H^{n-1}\ch_1(E)}{H^n\ch_0(E)}
\right) 
\] 
holds. 
Then for every $\alpha>0$ and 
a $\nu_{0, \alpha}$-semistable object $E$ 
with $\ch_0(E) \neq 0$, 
its Chern character satisfies the same inequality. 
\end{prop}
\begin{proof}
Assume for a contradiction that 
there exists a tilt-semistable object $E$ 
violating the required inequality. 
By \cite[Theorem 3.5]{bms16}, the object $E$ 
satisfies the usual BG inequality 
$\overline{\Delta}_H(E) \geq 0$. 
Hence we may assume that it has 
the minimum discriminant $\overline{\Delta}_H(E)$ 
among all tilt-semistable objects violating the inequality. 

Assume that $E$ becomes strictly $\nu_{0, \alpha_0}$-semistable 
for some $\alpha_0>0$. 
Then there exists a Jordan-H{\"o}lder factor $F$ of $E$ 
such that $p_H(F)$ is on the line segment connecting 
$p_H(E)$ and $(0, \alpha_0)$. 
Since the function $f$ is star-shaped, 
the object $F$ also violates the required inequality. 
Moreover, by \cite[Corollary 3.10]{bms16} we have 
$\overline{\Delta}_H(F) < \overline{\Delta}_H(E)$, 
which contradicts the minimality assumption on the discriminant. 

Now we can assume that $E$ is 
$\nu_{0, \alpha}$-semistable 
for all $\alpha \gg 0$. 
Hence by \cite[Lemma 2.7]{bms16}, 
the object $E$ satisfies one of the following conditions: 
\begin{enumerate}
\item $E \in \Coh(X)$ and it is $\mu_H$-semistable with $\ch_0(E)>0$. 

\item $\mcH^{-1}(E)$ is $\mu_H$-semistable, 
and $\dim \Supp \mcH^0(E) \leq n-2$.  
\end{enumerate}
In both cases, we get the contradiction by our assumption that 
$\mu_H$-semistable torsion free sheaves satisfy the desired inequality. 
\end{proof}

\subsection{Triple cover CY3}
Let us consider a general hypersurface 
\[
X:=X_{6} \subset P:=\bP(1, 1, 1, 1, 2) 
\]
of degree $6$ inside the weighted projective space. 
Then $X$ is a smooth projective Calabi-Yau threefold, 
which we call {\it triple cover CY3}. 
We will use 
general $(2, 2, 6)$, $(2, 6)$, $(2, 2)$-complete intersections 
\[
C_{2, 2, 6} \subset T_{2, 6} \subset X_6, 
\quad C_{2, 2, 6} \subset S_{2, 2}
\]
in $P$. 
Since the line bundle $\mcO_{P}(2)$ is free, they are smooth. 
The following are some of 
the numerical invariants of 
$C:=C_{2, 2, 6}$, $T:=T_{2, 6}$, $S:=S_{2, 2}$, and $X$. 

\begin{itemize}
\item $-K_{P}=6H_{P}$, $H_{P}^4=\frac{1}{2}$. 
\item $g(C)=25$, 
\item $-K_{S}=2H_{S}$, ($-K_{S})^2=8$. 
In particular, $S \cong \bP^1 \times \bP^1$. 
\item $C=3(-K_{S})$ as divisors in $S$. 
\item $\td_{S}=(1, H_S, 1)$. 
\item $K_{T}=2H_{T}$, $H_{T}^2=6$, 
$\td_{T}=(1, -H_T, 11)$. 
\item $\td_{X}=(1, 0, \frac{7}{6}H_{X}^2, 0)$, $H_{X}^3=3$. 
\end{itemize}

All the computations are straightforward. 
For example, to compute $\td_{X, 2}$, 
it is enough to compute $\chi(\mcO_X(1))$, 
which can be calculated using the exact sequence 
\[
0 \to \mcO_P(-5) \to \mcO_P(1) \to \mcO_X(1) \to 0. 
\]

\section{Clifford type theorem} \label{section:clifford}
Recall from the last subsection that 
we denote by 
\[
C=C_{2, 2, 6} \subset S=S_{2, 2} \subset P=\bP(1, 1, 1, 1, 2)
\]
the weighted complete intersections. 
We have $S \cong \bP^1 \times \bP^1$ 
and $C \in \left|\mcO_S(6, 6) \right|$. 
In this section, we will prove the following proposition:

\begin{prop} \label{prop:clifford}
Let $F$ be a slope stable vector bundle on $C$ 
of rank $r$, slope $\mu$. 
Put $t:=\mu/12$. 
Assume that $t \in [0, 1/2] \cup [3/2, 2]$. 
The following inequalities hold: 
\begin{enumerate}
\item When $t \in [0, 1/6)$, 
we have 
$h^0(F)/r \leq 
\frac{12t+24}{25}$.

\item When $t \in [1/6, 1/4)$, 
we have 
$h^0(F)/r \leq \max\left\{
\frac{8t+8}{9}, 
\frac{10}{19}t+\frac{145}{152}
\right\}$.

\item When $t \in [1/4, 1/2]$, 
we have 
$h^0(F)/r \leq \max\left\{
4t, \frac{33}{38}t+\frac{69}{76}
\right\}$.

\item When $t \in [3/2, 11/6]$, 
we have 
$h^0(F)/r \leq \max \left\{
4t, 
\frac{231}{32}t-\frac{375}{64} 
\right\}$.

\item When $t \in (11/6, \sqrt{14}/2]$, 
we have 
$h^0(F)/r \leq 
\frac{233t-191}{32}$. 

\item When $t \in [\sqrt{14}/2, 23/12]$, 
we have 
$h^0(F)/r \leq 
\frac{192t-168}{25}$. 

\item When $t \in [23/12, 2]$, 
we have 
$h^0(F)/r \leq 
12t-15$. 
\end{enumerate}
\end{prop}

\begin{figure}[htbp]
\begin{tabular}{c}
\begin{minipage}{0.5\hsize}
\begin{center}
\begin{tikzpicture}[xscale=10, yscale=2]
\draw[thick, ->] (0, 0.5)node[below]{$0$} 
-- (0.05+1/2, 0.5) node[above]{$t$}; 
\draw[thick, ->] (0, 0.5) -- (0, 2.5) node[right]{$h^0(F)/r$}; 
\node[left] at (0, 24/25) {$\frac{24}{25}$}; 
\draw[dashed] (1/6, 0.5)node[below]{$\frac{1}{6}$} -- (1/6, 12/25*1/6+24/25); 
\draw[domain=0:1/6] plot(\x, 12/25*\x+24/25); 
\draw[dashed] (1/4, 0.5)node[below]{$\frac{1}{4}$} -- (1/4, 33/38*1/4+69/76); 
\draw[domain=1/6:1/4] plot(\x, {max(8/9*\x+8/9, 10/19*\x+145/152)}); 
\draw[domain=1/4:1/2] plot(\x, {max(4*\x, 33/38*\x+69/76)}); 
\draw[dashed] (1/2, 0.5)node[below]{$\frac{1}{2}$} -- (1/2, 2); 
\draw[dashed] (0, 2)node[left]{$2$} -- (1/2, 2); 
\end{tikzpicture}
\end{center}
\end{minipage}

\begin{minipage}{0.5\hsize}
\begin{center}
\begin{tikzpicture}[xscale=10, yscale=1]
\draw[thick, ->] (3/2, 5.5)node[below]{$\frac{3}{2}$} 
-- (2.05, 5.5) node[right]{$t$}; 
\draw[thick, ->] (3/2, 5.5) -- (3/2, 9.5) node[right]{$h^0(F)/r$}; 
\node[left] at (3/2, 6) {$6$}; 
\draw[dashed] (11/6, 5.5)node[above left]{$\frac{11}{6}$} 
-- (11/6, 231/32*11/6-375/64); 
\draw[domain=3/2:11/6] plot(\x, {max(4*\x, 231/32*\x-375/64)}); 
\draw[domain=11/6:1.87] plot(\x, 233/32*\x-191/32); 
\draw[dashed] (1.87, 5.5)node[below]{$\frac{\sqrt{14}}{2}$} 
-- (1.87, 1.87*192/25-168/25); 
\draw[domain=1.87:23/12] plot(\x, 192/25*\x-168/25); 
\draw[dashed] (23/12, 5.5)node[below]{$\frac{23}{12}$} 
-- (23/12, 12*23/12-15); 
\draw[domain=23/12:2] plot(\x, 12*\x-15); 
\draw[dashed] (2, 5.5)node[below]{$2$} -- (2, 9); 
\draw[dashed] (3/2, 9)node[left]{$9$} -- (2, 9); 
\end{tikzpicture}
\end{center}
\end{minipage}

\end{tabular}
\caption{The strong Clifford type bounds on $C$.} 
\label{fig:Cliff}
\end{figure}
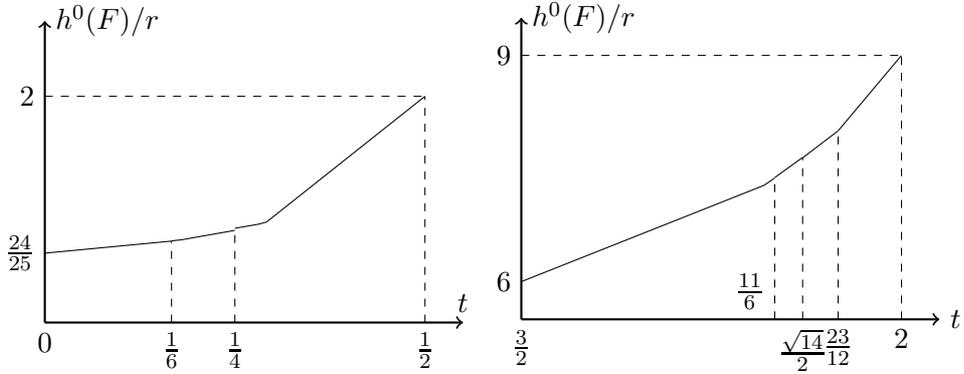

\begin{rmk}
The parameter $t=\mu/12$ 
naturally appears as the BN slope 
of the sheaf $\iota_*F$, 
where $\iota \colon C \hookrightarrow S$ 
is an embedding. 
Indeed, we have 
$\nu_{BN}(\iota_*F)=t-3$, 
as we will see 
in the proof of Lemma \ref{lem:bnslope} below. 

It is also compatible with 
the slope function on $T$ 
in the following sense. 
For a vector bundle $F$ on $T$, 
we have $t(F|_C)=\mu_{H_T}(F)$. 
\end{rmk}

Our strategy of the proof of Proposition \ref{prop:clifford}
is to use Bridgeland stability conditions on 
the surface $S$, with the following three steps. 
\begin{enumerate}
\item Regard $F$ as a torsion sheaf $\iota_{*}F \in \Coh(S)$, 
which is $\nu_{0, \alpha}$-stable for $\alpha \gg 0$. 
\item Estimate the first possible wall 
for $\iota_{*}F$ on the line $\beta=0$ 
in $(\alpha, \beta)$ plane, 
using the stronger form of the BG inequality on $S$. 
\item Bound global sections of BN-stable objects on $S$. 
\end{enumerate}

We define a function $\Upsilon$ on $\bR$ as 
\[
\Upsilon(x):=
\begin{cases}
\frac{1}{2}x^2-\frac{1}{2}(1-\{x\})^2 
	& (\{x\} \in (0, 1/2]) \\
\frac{1}{2}x^2-\frac{1}{2}\{x\}^2 
	& (\{x\} \in [1/2, 1)) \\
\frac{1}{2}x^2 & (\{x\}=0). 
\end{cases}
\]
Here, $\{x\}$ denotes the fractional part of $x \in \bR$. 
See Figure \ref{fig:bgonS} below 
for the shape of $\Upsilon$. 
The following stronger BG inequality 
on the quadric surface 
$S \cong \bP^1 \times \bP^1$ is well-known. 
We include a proof here, 
since it demonstrates the technique which 
we will frequently use in this section.

\begin{lem} \label{lem:BGdP}
Let $F$ be a slope semistable tosion free sheaf on $S$. 
Then we have an inequality 
\begin{equation} \label{eq:bgonS}
\frac{\ch_{2}(F)}{H^2\ch_{0}(F)} \leq 
\Upsilon\left(\mu_H(F) \right). 
\end{equation}
\end{lem}
\begin{proof}
Since we have 
$\Upsilon(x+1)=\Upsilon(x)+x+1/2$, 
the claim is invariant under 
tensoring with the line bundle $\mcO_S(H)$. 
Hence we may assume $\mu_H(F) \in (0, 1)$. 
By the stability of $F$ and the Serre duality, 
we have 
\[
\hom(\mcO(1), F)=0, \quad 
\ext^2(\mcO(1), F)=\hom(F, \mcO(-1))=0 
\]
and hence 
$0 \geq -\ext^1\left(\mcO(1), F \right) 
=\chi\left(\mcO(1), F \right)$. 
By computing the RHS using the Riemann-Roch theorem, 
we get the inequality 
\[
\frac{\ch_{2}(F)}{H^2\ch_{0}(F)} \leq 0. 
\]

On the other hand, 
again by the stability of $F$ 
and the Serre duality, 
we also have 
\[
\hom(F, \mcO)=0, \quad 
\ext^2(F, \mcO)=\hom(\mcO, F(-2))=0, 
\]
which imply the inequality 
$0 \geq -\ext^1(F, \mcO)=\chi(F, \mcO)$. 
Hence we obtain 
\[
\frac{\ch_2(F)}{H^2\ch_0(F)} \leq \mu_H(F)-\frac{1}{2}. 
\]

Taking the minimum, 
the inequality (\ref{eq:bgonS}) holds. 
\end{proof}

\begin{rmk}
In \cite{rud94}, Rudakov proved 
an inequality stronger 
than (\ref{eq:bgonS}). 
However, our inequality is already optimal 
at $\mu_H=1/2$ 
(consider $F=\mcO_S(1, 0)$). 
Because of this fact, 
we cannot improve our inequality 
in Theorem \ref{thm:mainintro} 
at $\mu_H=1/2$, 
even if we use the result in \cite{rud94}. 
\end{rmk}

We define a function 
$\widetilde{\Upsilon} \colon \bR \to \bR$ as 
\[
\widetilde{\Upsilon}(x):=
\begin{cases}
\Upsilon(x) & (|x| \in [0, 1]) \\
\max\left\{ 
	\Upsilon(x), 
	\frac{1}{2} \lfloor |x| \rfloor x
	\right\} 
& (|x| \geq 1). 
\end{cases}
\]
Here, $\lfloor |x| \rfloor$ 
denotes the integral part of 
the absolute value of 
a real number $x \in \bR$. 
Note that the function 
$\widetilde{\Upsilon}$ 
is star-shaped 
and we have $\Upsilon(x) \leq \widetilde{\Upsilon}(x)$ 
for all $x \in \bR$, 
see Figure \ref{fig:bgonS}.

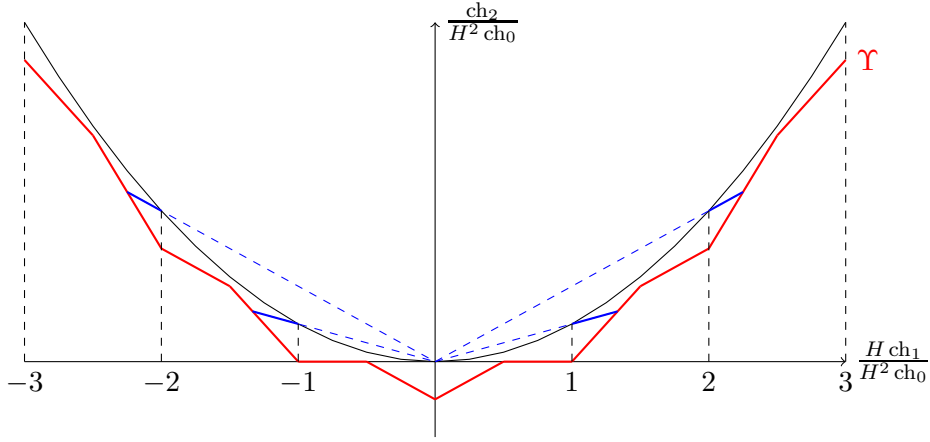
\begin{figure}[htbp]
\begin{center}
\begin{tikzpicture}[xscale=1.8]
\draw [->] (-3, 0) -- (3, 0) node[right]{$\frac{H\ch_1}{H^2\ch_0}$}; 
\draw[->] (0, -1) -- (0, 9/2) node[right]{$\frac{\ch_2}{H^2\ch_0}$}; 
\draw[red, thick, domain=0:1/2] plot(\x,\x-1/2); 
\draw[red, thick] (1/2, 0) -- (1, 0); 
\draw[red, thick, domain=1:3/2] plot(\x,2*\x-2); 
\draw[red, thick, domain=3/2:2] plot(\x, \x-1/2); 
\draw[red, thick, domain=2:5/2] plot(\x, 3*\x-9/2); 
\draw[red, thick, domain=5/2:3] 
plot(\x, 2*\x-2) node[right]{$\Upsilon$}; 
\draw[red, thick, domain=-3:-5/2] plot(\x, -2*\x-2); 
\draw[red, thick, domain=-5/2:-2] plot(\x, -3*\x-9/2); 
\draw[red, thick, domain=-2:-3/2] plot(\x, -\x-1/2); 
\draw[red, thick, domain=-3/2:-1] plot(\x,-2*\x-2); 
\draw[red, thick] (-1, 0) -- (-1/2, 0); 
\draw[red, thick, domain=-1/2:0] plot(\x,-\x-1/2); 
\draw[domain=-3:3] plot(\x, 1/2*\x*\x); 
\draw[blue, thick, domain=1:4/3] plot(\x, {abs(1/2*\x)}); 
\draw[blue, dashed, domain=0:4/3] plot(\x, {abs(1/2*\x)}); 
\draw[blue, thick, domain=-4/3:-1] plot(\x, {abs(1/2*\x)}); 
\draw[blue, dashed, domain=-4/3:0] plot(\x, {abs(1/2*\x)}); 
\draw[blue, thick, domain=2:9/4] plot(\x, {abs(\x)});
\draw[blue, dashed, domain=0:9/4] plot(\x, {abs(\x)}); 
\draw[blue, thick, domain=-9/4:-2] plot(\x, {abs(\x)}); 
\draw[blue, dashed, domain=-9/4:0] plot(\x, {abs(\x)}); 
\draw[dashed] (1, 0) node[below]{$1$} -- (1, 1/2); 
\draw[dashed] (-1, 0) node[below]{$-1$} -- (-1, 1/2); 
\draw[dashed] (2, 0) node[below]{$2$} -- (2, 2); 
\draw[dashed] (-2, 0) node[below]{$-2$} -- (-2, 2); 
\draw[dashed] (3, 0) node[below]{$3$} -- (3, 9/2); 
\draw[dashed] (-3, 0) node[below]{$-3$} -- (-3, 9/2); 
\end{tikzpicture}
\end{center}
\caption{The strong BG inequality $\Upsilon$ (red curve) on the quadric surface.
Blue lines show the modified curve $\widetilde{\Upsilon}$.} 
\label{fig:bgonS}
\end{figure}

We have the following consequences of 
Lemma \ref{lem:BGdP}. 

\begin{lem} \label{lem:geomstabS}
The following statements hold: 
\begin{enumerate}
\item Fix a positive real number $\alpha >0$. 
Let $F \in \Coh^0(S)$ 
be a $\nu_{0, \alpha}$-semistable object 
with $\ch_0(F) \neq 0$. 
Then the Chern character of $F$ satisfies 
the inequality 
\begin{equation} \label{eq:bgonS-tilt}
\frac{\ch_{2}(F)}{H^2\ch_{0}(F)} \leq 
\widetilde{\Upsilon}\left(\mu_H(F) \right). 
\end{equation}

\item For all real numbers 
$\beta, \alpha \in \bR$ with 
$\alpha > \Upsilon(\beta)$, 
the pair 
$\left(Z_{\beta, \alpha}, \Coh^\beta(S) \right)$ 
defines a stability condition on $D^b(S)$. 
Here the group homomorphism 
$Z_{\beta, \alpha} \colon K(S) \to \bC$ is defined as 
\[
Z_{\beta, \alpha}:=-\ch_2+\alpha H^2\ch_0
	+i\left(H\ch_1-\beta H^2\ch_0 \right). 
\]
\end{enumerate}
\end{lem}
\begin{proof}
(1) The first assertion follows from 
Lemma \ref{lem:BGdP} and Proposition \ref{prop:bg-tilt*}. 

(2) For the second assertion, 
we can apply the arguments 
in \cite{ab13, bri08} 
by replacing the classical BG inequality with 
the stronger one (\ref{eq:bgonS}). 
\end{proof}

In the next two lemmas, 
we control the position of the first possible wall 
for $\iota_*F$, where $F$ is a stable bundle on $C$, 
and then bound the slopes of the HN factors of $\iota_*F$ 
with respect to BN stability. 

\begin{lem} \label{lem:1stwall}
Let $F$ be a slope stable vector bundle on $C$ 
with rank $r$, slope $\mu$. 
Let $(\beta_1, \alpha_1), (\beta_2, \alpha_2)$, 
$\beta_1 < 0 < \beta_2$, 
be the end points of a wall for $\iota_*F$ 
with respect to $\nu_{\beta, \alpha}$-stability. 
Then we have $\beta_2-\beta_1 \leq 6$. 
\end{lem}
\begin{proof}
By the Grothendieck-Riemann-Roch theorem, we have 
\begin{equation} \label{eq:grr}
\ch(\iota_{*}F)=(0, 6rH, r(\mu-36)). 
\end{equation}
Suppose that there exists 
a positive integer $\alpha$ and 
a destabilizing sequence 
\[
0 \to F_{2} \to \iota_{*}F \to F_{1} \to 0
\]
in $\Coh^0(S)$ for $\nu_{0, \alpha}$-stability. 
Denote by $W$ the corresponding wall. 
Note that $F_2$ is a coherent sheaf. 
Let $T \subset F_2$ be a torsion part and put 
$Q:=F_2/T$. 
We have the following diagram 
in the tilted category $\Coh^0(S)$: 
\[
\xymatrix{
& &0 \ar[d] &0 \ar[d] & & \\
& &T \ar[d] \ar@{=}[r] &T \ar[d] & & \\
&0 \ar[r] &F_2 \ar[d] \ar[r] &\iota_*F \ar[d] \ar[r] &F_1 \ar[r] \ar@{=}[d] &0 \\
&0 \ar[r] &Q \ar[d] \ar[r] &\iota_*F/T \ar[d] \ar[r] &F_1 \ar[r] &0 \\
& &0 &0 & &
}
\]

By taking the $\Coh(S)$-cohomology 
of the bottom row in the above diagram, 
we get the exact sequence 
\[
0 \to Q/\mcH^{-1}(F_1) \to \iota_*F/T \to F_1 \to 0 
\]
in $\Coh(S)$. 
In particular, the sheaf 
$Q/\mcH^{-1}(F_1)$ 
is scheme-theoretically supported on the curve $C$. 
Hence we have a surjection 
$Q|_C \twoheadrightarrow Q/\mcH^{-1}(F_1)$ 
and so get an inequality 
\[
6H^2\ch_0(Q)=H\ch_1(Q|_C) \geq 
H\ch_1(Q/\mcH^{-1}(F_1)). 
\]
Note also that we have 
$\ch_0(Q)=\ch_0(\mcH^{-1}(F_1))$. 
Now we have 
\begin{equation} \label{eq:mudiff}
\begin{aligned}
\mu_H(Q)-\mu_H(\mcH^{-1}(F_1)) 
=\frac{H\ch_1(Q/\mcH^{-1}(F_1))}{H^2\ch_0(Q)} \leq 6. 
\end{aligned}
\end{equation}

Now let $(\beta_1, \alpha_1), (\beta_2, \alpha_2)$  
be the end points of the wall $W$ 
with $\beta_1 < 0 < \beta_2$. 
By Bertram's nested wall theorem 
(see e.g. \cite[Lemma 2.9]{li19b},  \cite{maci14}), 
we know that for $0 < \epsilon \ll 1$, 
we have 
\[
Q \in \Coh^{\beta_{2}-\epsilon}(S), \quad 
\mcH^{-1}(F_{1})[1] \in \Coh^{\beta_{1}+\epsilon}(S), 
\]
which in particular imply 
\[
\mu_{H}(Q) > \beta_{2}-\epsilon, \quad 
\mu_{H}(\mcH^{-1}(F_{1})) \leq \beta_{1}+\epsilon. 
\]
Combining with the inequality (\ref{eq:mudiff}), 
we have the desired inequality 
\[
\beta_{2}-\beta_{1} \leq 6. 
\]
\end{proof}

\begin{lem} \label{lem:bnslope}
Let $F$ be a slope stable vector bundle on $C$ 
with rank $r$, slope $\mu \in (0, 24)$. 
Let $t:=\mu/12$. 
The following statements hold: 
\begin{enumerate}
\item If $t \in \left(0, 2-\frac{\sqrt{14}}{2} \right]$, 
then the sheaf $\iota_*F$ is BN-stable. 

\item We have 
\[
\nu^+_{BN}(\iota_*F) \leq 
\begin{cases}
1-\frac{1}{2t} 
	& (t \in (2-\sqrt{14}/2, 1/2] \cup [3/2, \sqrt{14}/2]) \\
\frac{-9t+11}{-8t+7} 
	& (t \in [\sqrt{14}/2, 23/12]) \\
3t-5 
	& (t \in [23/12, 2]).
\end{cases}
\]

\item We have 
\[
\nu^-_{BN}(\iota_*F) \geq  
\begin{cases}
\frac{-5(2t-7)}{2(t-6)} 
	& (t \in (2-\sqrt{14}/2, 1/2] \cup [3/2, 11/6] ) \\
-2 
	& (t \in [11/6, 2]). 
\end{cases}
\]
\end{enumerate}
\end{lem}

\begin{proof}
Let $W$ be a wall for $\iota_*F$, 
and let $(\beta_1, \alpha_1), (\beta_2, \alpha_2)$  
be the end points of the wall $W$ 
with $\beta_1 < 0 < \beta_2$. 
Recall that the wall $W$ is a line segment 
with slope 
$\nu_{BN}(\iota_*F)=\mu/12-3=t-3$ 
(see (\ref{eq:grr}) for the second equality). 
Since the curve $\Upsilon$ is not continuous 
when $\beta \in \bZ$, 
the points $(\beta_i, \alpha_i)$ are 
either on the graph of $\Upsilon$, 
or on the vertical lines 
\[
L_n:=\left\{ 
(n, y) : \frac{n^2-1}{2} < y < \frac{n^2}{2} 
\right\}, 
\quad n \in \bZ.  
\]
When both of the end points 
$(\beta_i, \alpha_i)$ are 
on the curve $\Upsilon$, 
we say that the wall $W$ is 
of {\it Type A}, 
otherwise, we say it is of {\it Type B}.  

First assume that $W$ is of Type A. 
By Lemma \ref{lem:1stwall}, 
we know that $\beta_2-\beta_1 \leq 6$. 
Hence the slope of the line through 
$(\beta_{2}, \Upsilon(\beta_{2}))$ and 
$(\beta_{2}-6, \Upsilon(\beta_{2}-6))$ 
is smaller than or equal to that of $W$, i.e., 
$\beta_{2}-3 \leq t-3$. 
We conclude that every Type A wall is 
below the line 
$y=(t-3)(x-t)+\Upsilon(t)$. 

On the other hand, 
for a given point $p=(\beta, \alpha) \in L_n$, 
let $W_p$ be the line passing through 
the point $p$ with slope $t-3$. 
It is easy to compute the intersection points 
of $W_p$ and $\Upsilon \cup \bigcup_{n \in \bZ}L_n$. 
Together with the constraint $\beta_2-\beta_1 \leq 6$, 
we can find the first possible wall of Type B. 

Using these observations, 
we can list up the equation of 
the first possible wall: 
\begin{itemize}
\item When $t \in [0, 1/2]$, 
the following is the first possible wall 
\[
y=(t-3)(x-t)+t-1/2, 
\]
Note that if $t \in [0, 2-\sqrt{14}/2]$, 
it is negative at $x=0$, 
hence $\iota_*F$ is BN stable. 

\item When $t \in [3/2, 2]$, 
one of the following is the first possible wall 
\[
y=(t-3)(x-t)+t-1/2, \quad 
y=(t-3)(x+4)+8. 
\]
The first one is the line 
passing through the points 
$(t, \Upsilon(t))$ and 
$(t-6, \Upsilon(t-6))$. 
The second one is the line 
with slope $t-3$, passing through the point 
$(-4, \Upsilon(-4))$. 
See Figure \ref{fig:BGonT} below. 
\end{itemize}

Let $L$ be the first possible wall 
described above, 
and let $(\beta_{\max}, \alpha_{\max})$, 
$(\beta_{\min}, \alpha_{\min})$ 
be the intersection points of $L$ 
with the curve $\Upsilon$ 
with $\beta_{\min}<\beta_{\max}$. 
Then any wall $W$ should be below the line $L$. 
Now consider the maximal destabilizing subobject 
$E_{1} \subset \iota_{*}F$ 
with respect to the BN stability. 
We have three numerical constraints on $E_{1}$: 
\begin{itemize}
\item Since $\iota_{*}F$ is $\nu_{\alpha, 0}$-stable 
for $\alpha$ sufficiently large, 
we have $\ch_{0}(E_{1}) >0$. 

\item $E_{1}$ satisfies 
the BG type inequality (\ref{eq:bgonS-tilt}). 

\item The point $p_{H}(E_{1})$ is below the line $L$. 
\end{itemize}

Among all points satisfying the above three conditions, 
its slope becomes maximum at the point 
$(\alpha_{\max}, \beta_{\max})$, 
hence we get the bound 
\[
\nu_{BN}(E_{1}) \leq \frac{\alpha_{\max}}{\beta_{\max}}. 
\]

Now the straightforward computation shows the result. 
Similarly we can get the bound 
$\nu^-_{BN}(\iota_{*}F) \geq \alpha_{\min}/\beta_{\min}$. 

\end{proof}

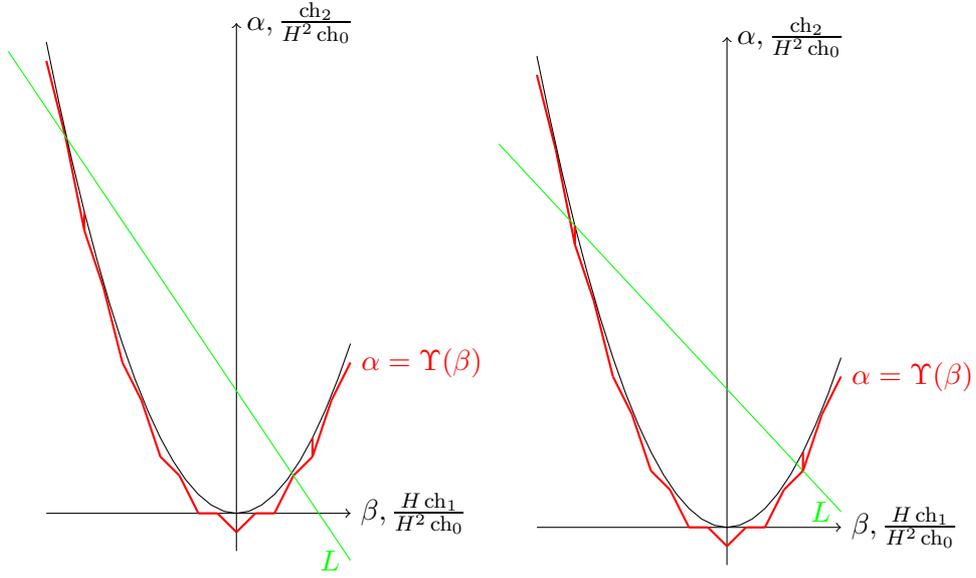
\begin{figure}[htbp]
\begin{tabular}{c}
\begin{minipage}{0.5\hsize}
\begin{center}
\begin{tikzpicture}[scale=0.5]
\draw [->] (-5, 0) -- (3, 0) node[right]{$\beta, \frac{H\ch_1}{H^2\ch_0}$}; 
\draw[->] (0, -1) -- (0, 13) node[right]{$\alpha, \frac{\ch_2}{H^2\ch_0}$}; 
\draw[red, thick, domain=0:1/2] plot(\x,\x-1/2); 
\draw[red, thick] (1/2, 0) -- (1, 0); 
\draw[red, thick, domain=1:3/2] plot(\x,2*\x-2); 
\draw[red, thick, domain=3/2:2] plot(\x, \x-1/2); 
\draw[red, thick, domain=2:5/2] plot(\x, 3*\x-9/2); 
\draw[red, thick, domain=5/2:3] plot(\x, 2*\x-2) 
node[right]{$\alpha=\Upsilon(\beta)$}; 
\draw[red, thick, domain=-5:-9/2] plot(\x, -4*\x-8); 
\draw[red, thick, domain=-9/2:-4] plot(\x, -5*\x-25/2); 
\draw[red, thick, domain=-4:-7/2] plot(\x, -3*\x-9/2); 
\draw[red, thick, domain=-7/2:-3] plot(\x, -4*\x-8); 
\draw[red, thick, domain=-3:-5/2] plot(\x, -2*\x-2); 
\draw[red, thick, domain=-5/2:-2] plot(\x, -3*\x-9/2); 
\draw[red, thick, domain=-2:-3/2] plot(\x, -\x-1/2); 
\draw[red, thick, domain=-3/2:-1] plot(\x,-2*\x-2); 
\draw[red, thick] (-1, 0) -- (-1/2, 0); 
\draw[red, thick, domain=-1/2:0] plot(\x,-\x-1/2); 
\draw[domain=-5:3] plot(\x, 1/2*\x*\x); 
\draw[green, domain=-6:3] plot(\x, -3/2*\x+13/4) node[left]{$L$}; 
\draw[red, thick] (-4, 15/2) -- (-4, 8); 
\draw[red, thick] (2, 2) -- (2, 3/2); 
\end{tikzpicture}
\end{center}
\end{minipage}

\begin{minipage}{0.5\hsize}
\begin{center}
\begin{tikzpicture}[scale=0.5]
\draw [->] (-5, 0) -- (3, 0) node[right]{$\beta, \frac{H\ch_1}{H^2\ch_0}$}; 
\draw[->] (0, -1) -- (0, 13) node[right]{$\alpha, \frac{\ch_2}{H^2\ch_0}$}; 
\draw[red, thick, domain=0:1/2] plot(\x,\x-1/2); 
\draw[red, thick] (1/2, 0) -- (1, 0); 
\draw[red, thick, domain=1:3/2] plot(\x,2*\x-2); 
\draw[red, thick, domain=3/2:2] plot(\x, \x-1/2); 
\draw[red, thick, domain=2:5/2] plot(\x, 3*\x-9/2); 
\draw[red, thick, domain=5/2:3] plot(\x, 2*\x-2) 
node[right]{$\alpha=\Upsilon(\beta)$}; 
\draw[red, thick, domain=-5:-9/2] plot(\x, -4*\x-8); 
\draw[red, thick, domain=-9/2:-4] plot(\x, -5*\x-25/2); 
\draw[red, thick, domain=-4:-7/2] plot(\x, -3*\x-9/2); 
\draw[red, thick, domain=-7/2:-3] plot(\x, -4*\x-8); 
\draw[red, thick, domain=-3:-5/2] plot(\x, -2*\x-2); 
\draw[red, thick, domain=-5/2:-2] plot(\x, -3*\x-9/2); 
\draw[red, thick, domain=-2:-3/2] plot(\x, -\x-1/2); 
\draw[red, thick, domain=-3/2:-1] plot(\x,-2*\x-2); 
\draw[red, thick] (-1, 0) -- (-1/2, 0); 
\draw[red, thick, domain=-1/2:0] plot(\x,-\x-1/2); 
\draw[domain=-5:3] plot(\x, 1/2*\x*\x); 
\draw[green, domain=-6:3] plot(\x, -13/12*\x+11/3) node[left]{$L$}; 
\draw[red, thick] (-4, 15/2) -- (-4, 8); 
\draw[red, thick] (2, 2) -- (2, 3/2); 
\end{tikzpicture}
\end{center}
\end{minipage}

\end{tabular}
\caption{The first possible wall $L$ 
when $t=3/2$ (left) and 
$t=23/12$ (right).} \label{fig:BGonT}
\end{figure}

The following lemma gives the upper bound 
on the number of global sections for 
BN stable objects. 
\begin{lem} \label{lem:sectionbn}
Let $F \in \Coh^0(S)$ be a BN stable object. 
Then the following inequalities hold: 
\begin{itemize}
\item When $-1 < \nu_{BN}(F) < +\infty$, we have 
\[
\hom(\mcO_{S}, F) = \ch_{0}(F) + H\ch_{1}(F) + \ch_{2}(F). 
\]

\item When $\nu_{BN}(E) \in (-n-1, -n)$, $n \in \bZ_{>0}$, 
we have 
\[
\hom(\mcO_{S}, F) 
	\leq \ch_{0}(F) + \frac{1}{2n+1}H\ch_{1}(F) 
		+ \frac{1}{(2n+1)^2}\ch_{2}(F). 
\]

\item When $\nu_{BN}(E)=-n$, 
$n \in \bZ_{>0}$, we have 
\[
\hom(\mcO_S, F) \leq \ch_0(F)+\frac{1}{4n}H\ch_1(F). 
\] 
\end{itemize}
\end{lem}
\begin{proof}
First assume that $\nu_{BN}(F) > -1$. 
Noting 
$\nu_{BN}(\mcO_S[1])=+\infty$ and 
$\nu_{BN}(\mcO_{S}(-2)[1])=-1$, 
we have the following vanishings for any $i \geq 0$: 
\begin{align*}
&\hom(\mcO_{S}, F[1+i])=\hom(F, \mcO_{S}(-2)[1-i])=0, \\
&\hom(\mcO_{S}, F[-1-i])=\hom(\mcO_{S}[1+i], F)=0. 
\end{align*}
Hence by the Riemann-Roch, we get 
\begin{align*}
\hom(\mcO_{S}, F)=\chi(F)
&=\int_{S}\ch(F). (1, H, 1) \\
&=\ch_{0}(F) + H\ch_{1}(F) +\ch_{2}(F). 
\end{align*}

Next consider the case 
$\nu_{BN}(E) \in (-n-1, -n)$, $n \in \bZ_{>0}$. 
Let 
\[
(x_{F}, y_{F}):=p_{H}(F)
=\left(
\frac{H\ch_{1}(F)}{H^2\ch_{0}(F)}, 
\frac{\ch_{2}(F)}{H^2\ch_{0}(F)} 
\right). 
\]
Since we assumed 
$y_{F}/x_{F}=\nu_{BN}(F) \leq -1<0$, 
the line through $(0, 0)$ and $(x_{F}, y_{F})$ 
intersects with the region 
$y \geq 1/2x^2, x<0$. 

Take such a point $(\beta, \alpha)$. 
Then we know that the objects 
$F, \mcO_S \in \Coh^\beta(S)$ are 
$\nu_{\beta, \alpha}$-semistable with 
\[
\nu_{\beta, \alpha}(F)
=\nu_{\beta, \alpha}(\mcO_{S})
=\alpha/\beta=\nu_{BN}(F). 
\]
Let us consider the exact triangle 
\[
\Hom(\mcO_{S}, F) \otimes \mcO_{S} \xrightarrow{ev} 
F \to \widetilde{F}:=\Cone(ev). 
\]
Since the only Jordan-H\"{o}lder factor 
of $\Hom(\mcO_S, F) \otimes \mcO_S$ 
with respect to $\nu_{\beta, \alpha}$-stability 
is $\mcO_S$, the evaluation map 
$ev$ must be injective in the category $\Coh^{\beta}(S)$. 
Hence it follows that 
$\widetilde{F} \in \Coh^{\beta}(S)$ 
and it is  $\nu_{\beta, \alpha}$-semistable with 
$\nu_{\beta, \alpha}(\widetilde{F})
=\nu_{\beta, \alpha}(F)=\nu_{BN}(F)$. 
Now choose $\beta$ sufficiently close to zero so that 
$\mcO_{S}(-2n)[1] \in \Coh^{\beta}(S)$. 
As before, we have the vanishing statements 
\begin{align*}
&\hom(\mcO_{S}(-2n), \widetilde{F}[1+i])
=\hom(\widetilde{F}, \mcO_{S}(-(2n+2))[1-i])=0, \\
&\hom(\mcO_{S}(-2n), \widetilde{F}[-1-i])
=\hom(\mcO_{S}(-2n)[1+i], \widetilde{F})=0 
\end{align*}
for $i \geq 0$. 
Hence we have 
\begin{align*}
0 &\leq \hom(\mcO_{S}(-2n), \widetilde{F}) \\
	&=\chi(\mcO_{S}(-2n), \widetilde{F}) \\
	&=\ch_{2}(F)+(2n+1)H\ch_{1}(F)
		+(2n+1)^2(\ch_{0}(F)-\hom(\mcO_{S}, F)), 
\end{align*}
and so 
\[
\hom(\mcO_{S}, F) \leq 
	\ch_{0}(F)+\frac{H\ch_{1}(F)}{2n+1} 
	+\frac{\ch_{2}(F)}{(2n+1)^2}
\]
as required. 

Finally, assume that 
$\nu_{BN}(F)=-n$, 
$n \in \bZ_{>0}$. 
Then the same argument shows that 
$\chi(\mcO_S(-2n+1), \widetilde{F}) \geq 0$, 
and we get the inequality 
\[
\hom(\mcO_S, F) \leq 
\ch_0(F)+\frac{H\ch_1(F)}{2n}+\frac{\ch_2(F)}{(2n)^2}
=\ch_0(F)+\frac{1}{4n}H\ch_{1}(F). 
\]
\end{proof}

Let us define a function 
$\Omega \colon \bR \times \bR_{>0} \to \bR_{>0}$ as 
\[
\Omega(x, y):=
\begin{cases}
y+x & (x/y > -1) \\
\frac{y}{2n+1}+\frac{x}{(2n+1)^2} & (x/y \in (-n-1, -n), n \in \bZ_{>0}) \\
\frac{1}{4n}y & (x/y=-n, n \in \bZ_{>0}). 
\end{cases}
\]

\begin{lem} \label{lem:n=2}
Let $O \in \bR^2$ be the origin, 
let $P=(x_p, y_p), Q=(x_q, y_q) \in \bR \times \bR_{>0}$ be points satisfying 
$x_p/y_p < x_q/y_q$ and $y_p > y_q$. 

Among all the sequences $O=P_0, P_1, \cdots, P_{m-1}, P_m=P$ of points 
in the triangle $OPQ$ such that $P_0P_1 \cdots P_m$ forms a convex polygon, 
the sum 
\[
\sum_{i=1}^m \Omega(\overrightarrow{P_{i-1}P_i})
\]
can achieve the maximum only when $m \leq 2$. 

Moreover, when $m=2$, the point $P_1=(x_1, y_1)$ can be chosen to satisfy one of the following conditions: 
\begin{itemize}
\item $P_1=Q$, 
\item $P_1$ is on the line segment $OQ$ (resp. $PQ$) 
such that the slope of $P_1P$ (resp. $OP_1$) 
is $-1/n$ for some $n \in \bZ_{>0}$. 
\item the lines $OP_1$ and $P_1P$ 
have slopes $-1/m, -1/n$ 
for some integers $m, n \in \bZ_{>0}$. 
\end{itemize}
\end{lem}
\begin{proof}
The proof is elementary and almost identical with that of \cite[Lemma 4.11]{li19b}. 
The key observations are the following: 
\begin{itemize}
\item The function $\Omega$ is linear with respect to both variables $x$ and $y$ 
as long as the slope $x/y$ is fixed. 
\item The function $\Omega$ is upper semi-continuous. 
\end{itemize}

We refer to \cite[Lemma 4.11]{li19b} for the details. 
\end{proof}

Now we can prove Proposition \ref{prop:clifford}. 
\begin{proof}[Proof of Proposition \ref{prop:clifford}]
Let $F$ be a slope stable vector bundle on $C$ 
of rank $r$ and slope $\mu$. 
Let 
\[
0=E_{0} \subset E_{1} \subset \cdots \subset E_{m}=\iota_{*}F
\]
be the HN filtration with respect to the BN stability, 
and define $P_{i}:=(\ch_{2}(E_{i}), H\ch_{1}(E_{i}))$. 
We have an inequality 
\begin{equation} \label{eq:Omegaineq}
h^0(F) \leq 
\sum_{i=1}^m \Omega(\overrightarrow{P_{i-1}P_{i}})
\end{equation}
by Lemma \ref{lem:sectionbn} 
(cf. \cite[Equation (21)]{li19b}). 
We will bound the RHS in the above inequality. 
Let us put 
\[
P=(x_p, y_p):=(\ch_2(\iota_{*}F), H\ch_{1}(\iota_{*}F))
=(r(\mu-36), 12r), 
\]
and $Q=(x_q, y_q)$ to be a point 
such that $x_q/y_q$ is the upper bound for 
$\nu^+_{BN}(\iota_{*}F)$, 
and $(x_p-x_q)/(y_p-y_q)$ is the lower bound for 
$\nu^-_{BN}(\iota_{*}F)$, given in Lemma \ref{lem:bnslope}. 
We know that the HN polygon of $\iota_{*}F$ with respect to the BN stability 
is inside the triangle $OPQ$. 
Hence by Lemma \ref{lem:n=2}, we may assume $m=2$ and 
the point $P_1$ satisfies one of the following conditions: 
\begin{itemize}
\item $P_1=Q$, 
\item $P_1$ is on the line segment $OQ$ (resp. $PQ$) 
such that the slope of $P_1P$ (resp. $OP_1$) 
is $-1/n$ for some $n \in \bZ_{>0}$. 
\item the lines $OP_1$ and $P_1P$ 
have slopes $-1/m, -1/n$ 
for some integers $m, n \in \bZ_{>0}$. 
\end{itemize}

We now argue case by case. 

(0) When $t \in (0, 2-\sqrt{14}/2]$, 
the sheaf $\iota_*F$ is BN stable. 
The BN slope is 
$\nu_{BN}(\iota_*F)=t-3 \in (-3, -2)$, 
hence by Lemma \ref{lem:sectionbn}, 
we have 
\[
h^0(F)/r \leq 
\frac{12}{5}+\frac{12(t-3)}{25}
=\frac{12t+24}{25}. 
\]

(1) Assume $t \in (2-\sqrt{14}/2, 1/6)$. 
By Lemma \ref{lem:bnslope}, we have 
\begin{itemize}
\item The slope of $\overrightarrow{OP}$ is 
$\frac{1}{t-3} \in (-1/2, -1/3)$, 

\item The slope of $\overrightarrow{OQ}$ is 
$\frac{2t}{2t-1} \in (-1/2, -1/3)$. 

\item The slope of $\overrightarrow{QP}$ is 
$\frac{2(t-6)}{-5(2t-7)} \in (-1/2, -1/3)$. 
\end{itemize}
Hence we may assume $P_1=Q=(x_q, y_q)$. 
We get 
\begin{align*}
h^0(F) &\leq 
\Omega(\overrightarrow{OQ})+\Omega(\overrightarrow{QP}) \\
&=\frac{y_q}{5}+\frac{x_q}{25}
+\frac{y_p-y_q}{5}+\frac{x_p-x_q}{25} 
=\frac{12t+24}{25}r. 
\end{align*}

(2) Assume $t \in [1/6, 1/4)$. 
Then the slopes of the triangle $OPQ$ are 
the same as the case (1). 
The only difference is that 
the slope of $\overrightarrow{OQ}$ 
sits inside the interval $(-1, -1/2]$, 
instead of $(-1/2, -1/3)$. 
Hence we may take $P_1$ as $Q$ 
or the point $A$ on the line segment $QP$ 
with slope $-1/2$. 
The coordinates are given as 
\[
Q=((2t-1)r, 2tr), \quad 
A=\left(
	\frac{24(2t^2-8t+1)}{-6t+11}r, 
	\frac{12(2t^2-8t+1)}{-6t+11}r
	\right). 
\]
First consider the case of $P_1=Q \neq A$. 
We have 
\begin{align*}
\Omega(\overrightarrow{OQ})+\Omega(\overrightarrow{QP})
&=\frac{y_q}{3}+\frac{x_q}{9}
+\frac{y_p-y_q}{5}+\frac{x_p-x_q}{25} \\
&=\frac{8t+8}{9}r. 
\end{align*}

When $P_1=A$, we have 
\begin{align*}
\Omega(\overrightarrow{OA})+\Omega(\overrightarrow{AP})
&=\frac{y_a}{8}
+\frac{y_p-y_a}{5}+\frac{x_p-x_a}{25} \\
&=\frac{1}{200}y_a+\frac{12t+24}{25}r. 
\end{align*}

As a function on $t \in [1/6, 1/4]$, 
we have an inequality 
\[
y_a(t) \leq 
\left(\frac{176}{19}t
-\frac{23}{19}\right)r 
\]
since the equality hold for $t=1/4, 1/6$, 
and $y^{''}_a(t) > 0$ for $t < 11/6$. 
Hence we obtain 
\begin{align*}
\Omega(\overrightarrow{OA})+\Omega(\overrightarrow{AP}) 
&\leq \frac{1}{200}\left(
	\frac{176}{19}t-\frac{23}{19}
	\right)r
+\frac{12t+24}{25}r \\
&=\left(\frac{10}{19}t+\frac{145}{152}\right)r. 
\end{align*}

We conclude that 
\[
h^0(F)/r \leq \max\left\{
\frac{8t+8}{9}, 
\frac{10}{19}t+\frac{145}{152}
\right\}. 
\]

(3) Assume $t \in [1/4, 1/2]$. 
Again the only difference with the cases (1), (2) 
is that the slope of $OQ$ is 
smaller than or equal to $-1$ 
(or $+\infty$ when $t=1/2$) 
in the present case. 
Hence we may choose $P_1$ to be $Q$, 
or the points $A, B$ on the line segment $QP$ 
with slope $-1/2, -1$, respectively. 
The coordinate of $Q, A$ 
are the same as in (2), 
and we have 
\[
B=(x_b, y_b)=\left(
	-\frac{12(2t^2-8t+1)}{8t-23}r, 
	\frac{12(2t^2-8t+1)}{8t-23}r
	\right). 
\]
We get 
\begin{align*}
&\Omega(\overrightarrow{OQ})+\Omega(\overrightarrow{QP})
=y_q+x_q
+\frac{y_p-y_q}{5}+\frac{x_p-x_q}{25} 
=4tr, \\
&\Omega(\overrightarrow{OB})+\Omega(\overrightarrow{BP})
=\frac{1}{4}y_b 
+\frac{y_p-y_b}{5}+\frac{x_p-x_b}{25} 
=\frac{9}{100}y_b+\frac{12t+24}{25}r. 
\end{align*}

As a function of $t \in [1/4, 1/2]$, 
we have the inequality 
\[
y_b(t) \leq 
\left(\frac{82}{19}t-\frac{11}{19} \right)r 
\]
and hence 
\[
\Omega(\overrightarrow{OB})+\Omega(\overrightarrow{BP})
\leq \frac{33}{38}tr+\frac{69}{76}r. 
\]
We can directly compute that 
\[
\Omega(\overrightarrow{OA})+\Omega(\overrightarrow{AP})
=\frac{1}{200}y_a+\frac{12t+24}{25} 
\leq \frac{33}{38}tr+\frac{69}{76}r. 
\]
We conclude that 
\[
h^0(F)/r \leq \max\left\{
4t, \frac{33}{38}t+\frac{69}{76}
\right\}. 
\]

(4) Assume $t \in [3/2, 11/6]$. 
In this case, we have 
\begin{itemize}
\item The slope of $\overrightarrow{OP}$ is 
$\frac{1}{t-3} \in (-1, -1/2)$, 

\item The slope of $\overrightarrow{OQ}$ is 
$\frac{2t}{2t-1} >0$, 

\item The slope of $\overrightarrow{QP}$ is 
$\frac{2(t-6)}{-5(2t-7)} \in [-1/2, -1/3)$. 
\end{itemize}

There are four choices of the point $P_1$, 
say $Q, A, B, C$, 
where 
$A$ is the point on the line $PQ$ 
with slope $-1$, 
$B$ is the point on the line $OQ$ 
such that the slope of $BP$ is $-1/2$, 
and $C$ is the intersection point of 
two lines $OA$ and $BP$. 
Explicitly, we have 
\begin{align*}
&Q=\left((2t-1)r, 2tr \right), \quad 
A=\left(
	-\frac{12(2t^2-8t+1)}{8t-23}r, 
	\frac{12(2t^2-8t+1)}{8t-23}r
	\right), \\
&B=\left(
	\frac{12(2t-1)(t-1)}{6t-1}r, 
	\frac{24t(t-1)}{6t-1}r
	\right), \quad 
C=\left(
-12(t-1)r, 12(t-1)r
\right). 
\end{align*}

Hence we get 
\begin{align*}
&\Omega(\overrightarrow{OQ})+\Omega(\overrightarrow{QP})
=\frac{4}{5}y_q+\frac{24}{25}x_q+\frac{12t+24}{25}r
=4tr, \\
&\Omega(\overrightarrow{OA})+\Omega(\overrightarrow{AP})
=\frac{9}{100}y_a+\frac{12t+24}{25}r, \\
&\Omega(\overrightarrow{OB})+\Omega(\overrightarrow{BP})
=y_b+x_b+\frac{y_p-y_b}{8}
=\frac{7}{8}y_b+x_b+\frac{3}{2}r, \\
&\Omega(\overrightarrow{OC})+\Omega(\overrightarrow{CP})
=\frac{y_c}{4}+\frac{y_p-y_c}{8}
=\frac{3}{2}tr. 
\end{align*}
Firstly, we can show that 
\[
\frac{9}{100}y_a < 
\frac{4}{5}y_q+\frac{24}{25}x_q 
\]
for $t \in [3/2, 11/6]$. 
On the other hand, 
It is easy to see 
\[
\Omega(\overrightarrow{OC})+\Omega(\overrightarrow{CP}) \leq 
\Omega(\overrightarrow{OB})+\Omega(\overrightarrow{BP})
=\frac{90t^2-96t+21}{2(6t-1)}r 
\leq \left(\frac{231}{32}t-\frac{375}{64} \right)r. 
\]

Hence we can conclude that 
\[
h^0(F)/r \leq \max \left\{
4t, 
\frac{231}{32}t-\frac{375}{64}
\right\}. 
\]

(5) Assume $t \in (11/6, \sqrt{14}/2]$. 
Then we have 
\begin{itemize}
\item The slope of $\overrightarrow{OP}$ is 
$\frac{1}{t-3} \in (-1, -1/2)$, 

\item The slope of $\overrightarrow{OQ}$ is 
$\frac{2t}{2t-1} >0$, 

\item The slope of $\overrightarrow{QP}$ is 
$-1/2$. 
\end{itemize}

Hence we can choose the point $P_1$ to be 
$B$ or $C$ appeared in the case (4) above. 
For $t \in [11/6, \sqrt{14}/2]$, we have 
\[
\Omega(\overrightarrow{OC})+\Omega(\overrightarrow{CP}) \leq 
\Omega(\overrightarrow{OB})+\Omega(\overrightarrow{BP})
=\frac{90t^2-96t+21}{2(6t-1)}r \leq 
\left(\frac{233}{32}t-\frac{191}{32} \right)r. 
\]

(6) Assume $t \in [\sqrt{14}/2, 23/12]$. 
In this case, we have 
\begin{itemize}
\item The slope of $\overrightarrow{OP}$ is 
$\frac{1}{t-3} \in (-1, -1/2)$, 

\item The slope of $\overrightarrow{OQ}$ is 
$\frac{8t-7}{9t-11} >0$, 

\item The slope of $\overrightarrow{QP}$ is 
$-1/2$. 
\end{itemize}

We may choose the point $P_1$ 
as $Q$ or $A$, 
where 
$A$ is the point on the line $PQ$ 
with slope $-1$. 
Explicitly, we have 
\begin{align*}
&Q=\left(
	\frac{12}{25}(9t-11)r, 
	\frac{12}{25}(8t-7)r 
	\right), \quad 
A=\left(
	-12(t-1)r, 
	12(t-1)r
	\right), 
\end{align*}

and hence 
\begin{align*}
&\Omega(\overrightarrow{OQ})+\Omega(\overrightarrow{QP})
=y_q+x_q+\frac{y_p-y_q}{8}
=\frac{192t-168}{25}r, \\
&\Omega(\overrightarrow{OA})+\Omega(\overrightarrow{AP})
=\frac{y_a}{4}+\frac{y_p-y_a}{8}
=\frac{3}{2}tr. \\
\end{align*}
We can see that 
$\Omega(\overrightarrow{OQ})+\Omega(\overrightarrow{QP}) 
\geq \Omega(\overrightarrow{OA})+\Omega(\overrightarrow{AP})$, 
hence we conclude that 
\[
h^0(F)/r \leq
\frac{192t-168}{25}. 
\]

(7) Assume $t \in [23/12, 2)$. Then we have 
\begin{itemize}
\item The slope of $\overrightarrow{OP}$ is 
$\frac{1}{t-3} \in (-1, -1/2)$, 

\item The slope of $\overrightarrow{OQ}$ is 
$\frac{1}{3t-5} >0$, 

\item The slope of $\overrightarrow{QP}$ is 
$-1/2$. 
\end{itemize}
Hence $P_{1}=Q$ or $A$, 
where $A$ is 
the point on the line segment $PQ$ with slope $-1$, i.e., 
\[
Q=((12t-20)r, 4r), \quad 
A=((-12t+12)r, (12t-12)r). 
\]
We can calculate as 
\begin{align*}
&\Omega(\overrightarrow{OQ})+\Omega(\overrightarrow{QP}) 
=\left(12t-15 \right)r, \\
&\Omega(\overrightarrow{OA})+\Omega(\overrightarrow{AP})
=\frac{3}{2}tr
\end{align*}
hence taking the maximum, we conclude that 
\[
h^0(F)/r \leq 
12t-15. 
\]
\end{proof}

\section{Stronger BG inequality} \label{section:strongBG}
Using the Clifford type bound obtained 
in Proposition \ref{prop:clifford}, 
we prove the following stronger version of 
the (classical) BG inequality on a triple cover CY3 
$X:=X_6 \subset \bP(1, 1, 1, 1, 2)$. 
\begin{thm} \label{thm:strongBG}
Let $X$ be a triple cover CY3. 
Let $F \in D^b(X)$ 
be a $\nu_{\alpha, 0}$-semistable object 
for some $\alpha > 0$, 
with $\mu_H(F) \in [-1, 1]$. 
Then we have the following inequality 
\begin{equation} \label{eq:bgonT}
\frac{H\ch_2(F)}{H^3\ch_0(F)} 
\leq \Xi \left(\left|\mu_H(F) \right|\right), 
\end{equation}
where 
\[
\Xi(t):=\left\{ \begin{array}{ll}
t^2-t & (t \in [0, 1/4]) \\ 
3t/4-3/8 & (t \in [1/4, 1/2]) \\
t/4-1/8 & (t \in [1/2, 3/4]) \\
t^2-1/2 & (t \in [3/4, 1]). 
\end{array} \right.
\]
\end{thm}
\begin{proof}
Assume for contradiction that 
there is a tilt semistable object $F$ 
violating the inequality (\ref{eq:bgonT}). 
We may assume $\mu_H(F) \geq 0$ 
by replacing $F$ with $F^\vee$ if necessary. 
First observe that the following conditions hold: 
\begin{itemize}
\item Let $p=(a, b)$ be an arbitrary point 
with $a \in [0, 1], b>\Xi(a)$, 
and take a real number 
$\alpha>0$ 
(resp. $\alpha'>1/2$). 
Then the line segment connecting the points 
$p$ and $(0, \alpha)$ 
(resp. $(1, \alpha')$) 
is above 
the graph of $\Xi$. 

\item Let $L$ be the line through $p_H(F)$ and $p_H(F(-2H)[1])$. 
Then $L$ passes through points $(0, \alpha_0), (-1, \alpha_0')$ 
with $\alpha_0 >0, \alpha_0' > 1/2$.  
Putting $(a, b):=p_H(F)$, 
the conditions are equivalent to the inequalities 
\[
b > a^2-a, \quad 
b > a^2-1/2. 
\]
\end{itemize}

Under these conditions, 
we can apply the arguments in 
\cite[Proposition 5.2, Corollary 5.4]{li19b}. 
As a result, by restricting to 
the surface $T=T_{2, 6} \subset X_{6}$, 
we obtain a tilt-stable object $F$ on $T$ 
with $\mu_H(F) \in (0, 1)$ and 
\begin{equation} \label{eq:violate}
\frac{\ch_2(F)}{H^2\ch_0(F)} 
> \Xi \left(\frac{H\ch_1(F)}{H^2\ch_0(F)} \right).  
\end{equation}

Furthermore, by the first paragraph 
in the proof of \cite[Proposition 5.2]{li19b}, 
we may assume that 
\begin{itemize}
\item $\mu_{H_{T}}(F) \in (0, 1/2]$, 
\item $F$ is $\mu_{H_{T}}$-stable {\it coherent sheaf}, 
\item $F|_{C}, F^{\vee}(2H_{T})|_{C}$ are slope stable. 
\end{itemize}

Using the Riemann-Roch and the vanishings 
\[
\hom(\mcO_{T}, F(-2H_{T}))=0=\hom(\mcO_{T}, F^{\vee}) 
\]
(both follows from slope stability of $F$ 
and the assumption on its slope),  
we have 
\begin{equation} \label{eq:rr}
\begin{aligned} 
\ch_{2}(F)-H_{T}\ch_{1}(F)+11\ch_{0}(F) 
&=\chi(F) \\
&\leq h^0(F|_C)+h^0(F^{\vee}(2H_{T})|_{C}). 
\end{aligned}
\end{equation}
Note that we have 
\begin{align*}
&\ch(F |_C)=(\ch_0(F), 2H\ch_1(F)), \\
&\ch(F^{\vee}(2H_{T})|_C)=(\ch_0(F), 4H^2\ch_0(F)-2H\ch_1(F)). 
\end{align*}
Applying Proposition \ref{prop:clifford} 
to the RHS of (\ref{eq:rr}), 
we get 
\begin{equation} \label{eq:corcliff}
\frac{\ch_{2}(F)}{H^2\ch_{0}(F)} 
\leq \begin{cases}
-\frac{23}{25}\mu_H(F)-\frac{13}{75} 
& (\mu_H(F) \in (0, 1/12]) \\
-\frac{1}{5}\mu_H(F)-\frac{7}{30} 
& (\mu_{H}(F) \in [1/12, 2-\sqrt{14}/2]) \\ 
-\frac{641}{4800}\mu_H(F)-\frac{1157}{4800}
& (\mu_{H}(F) \in [2-\sqrt{14}/2, 1/6)) \\
-\frac{421}{3648}\mu_H(F)-\frac{595}{2432} 
& (\mu_{H}(F) \in [1/6, 89/496]) \\ 
-\frac{95}{1728}\mu_H(F)-\frac{883}{3456} 
& (\mu_{H}(F) \in [89/496, 37/206]) \\ 
\frac{13}{27}\mu_H(F)-\frac{19}{54} 
& (\mu_{H}(F) \in [37/206, 1/4)) \\ 
\frac{109}{228}\mu_H(F)-\frac{53}{152} 
& (\mu_{H}(F) \in [1/4, 69/238]) \\
\mu_H(F)-\frac{1}{2} 
& (\mu_{H}(F) \in [69/238, 1/2]) 
\end{cases} 
\end{equation}
In all cases, the inequalities (\ref{eq:corcliff}) 
contradict to the inequality (\ref{eq:violate}). 
\end{proof}

\section{The case of double cover} \label{section:2to1}
In this section, 
we consider the double cover $X$ of $\bP^3$ 
branched along a smooth hypersurface of degree $8$; 
$X$ is another example of Calabi-Yau threefolds. 
As in the previous sections, 
we treat $X$ as a weighted hypersurface 
in $P=\bP(1, 1, 1, 1, 4)$ of degree $8$. 
Let 
\[
C_{2, 4, 8} \subset T_{2, 8} \subset X_8, \quad 
C_{2, 4, 8} \subset S_{2, 4} 
\]
be smooth $(2, 4, 8)$-, $(2, 8)$-, $(2, 4)$-complete intersections 
in $P$. 
The following is the list of their numerical invariants we need. 

\begin{itemize}
\item $-K_P=8H$, $H_P^4=1/4$, 
\item $g(C)=49$, 
\item $S \cong \bP^1 \times \bP^1$, 
\item $K_T=2H_T$, $H_T^2=4$, $\td_T=(1, -H_T, 10)$. 
\item $\td_X=(1, 0, \frac{11}{6}H_X^2, 0)$, $H_X^3=2$. 
\end{itemize}

\begin{rmk}
To make the surface $S$ smooth, 
we take a $(2, 4)$-complete intersection 
instead of a $(2, 2)$-complete intersection. 
\end{rmk}

\subsection{Clifford type bound}
In this subsection, we will prove the Clifford type theorem 
for the curve $C=C_{2, 4, 8}$, 
using the embedding $\iota$ 
into the quadric surface $S \cong \bP^1 \times \bP^1$. 
\begin{lem} \label{lem:bnslope2}
Let $F$ be a slope stable vector bundle on $C$ 
with rank $r$ slope $\mu$. 
Put $t:=\frac{\mu}{16}$ and assume 
$t \in [0, 1/2] \cup [3/2, 2]$. 
Then we have the following 
statements hold: 
\begin{enumerate}
\item When 
$t \in \left[0, \frac{5-\sqrt{23}}{2} \right]$, 
the sheaf $\iota_*F$ is BN stable. 

\item We have 
\[
\nu^+_{BN}(\iota_{*}F) \leq 
\begin{cases}
1-\frac{1}{2t} & (t \in [\frac{5-\sqrt{23}}{2}, 1/2] \cup [3/2, \frac{\sqrt{23}-1}{2}]) \\ 
\frac{-13t+16}{-12t+11} & (t \in [\frac{\sqrt{23}-1}{2}, 31/16]) \\
4t-7  & (t \in [31/16, 2]). 
\end{cases}
\]

\item We have 
\[
\nu^-_{BN}(\iota_{*}F) \geq 
\begin{cases}
\frac{-7(2t-9)}{2(t-8)} & (t \in (\frac{5-\sqrt{23}}{2}, 1/2] \cup [3/2, 15/8]) \\
-3 & (t \in [15/8, 2]). 
\end{cases}
\]
\end{enumerate}
\end{lem}
\begin{proof}
The proof is almost identical to that of Lemma \ref{lem:bnslope}. 
Hence we just give an outline of the proof. 
Let us consider the embedding 
$\iota \colon C \hookrightarrow S$. 
For $F \in \Coh(C)$, we have 
$\ch(\iota_*F)=(0, 8rH, r(\mu-64))$. 
Let $W$ be a wall for $\iota_*F$ 
with respect to the $\nu_{\alpha, 0}$-stability, 
and let $\beta_1, \beta_2$ be the $\beta$-coordinates 
of the end points of the wall  $W$ 
with $\beta_1 < 0 < \beta_2$. 
Then we can show that $\beta_2-\beta_1 \leq 8$. 
We have $\nu_{BN}(\iota_*F)=t-4$, 
and we can get the bounds of the first possible wall 
as follows: 
\begin{itemize}
\item When $t \in [0, 1/2]$, 
the equation of the first possible wall is 
\[
y=(t-4)(x-t)+t-1/2, 
\]
which is the line passing through the points 
$(t, \Upsilon(t)), (t-8, \Upsilon(t-8))$. 
We can see that $y(0) \leq 0$ 
for $t \in [0, \frac{5-\sqrt{23}}{2}]$, 
hence the sheaf $\iota_*F$ 
is BN stable. 

\item When $t \in [3/2, 2]$, 
we have two possibilities of the first wall: 
\[
y=(t-4)(x-t)+t-1/2, \quad 
y=(t-4)(x+6)+18. 
\]
The first equation is 
the line passing through the points 
$(t, \Upsilon(t)), (t-8, \Upsilon(t-8))$, 
and the second one is the line 
passing through the point $(-6, \Upsilon(-6))$ 
with slope $t-4$. 
\end{itemize}

As similar to Lemma \ref{lem:bnslope}, 
we get the bound on 
$\nu_{BN}^{\pm}(\iota_*F)$ 
by computing the end points 
of the first possible walls listed above. 
\end{proof}

We get the following Clifford type bound: 
\begin{prop} \label{prop:clifford2}
Let $F$ be a slope stable vector bundle on $C$ 
of rank $r$, slope $\mu \in (0, 8] \cup [24, 32)$. 
Put $t:=\mu/16$. 
The following inequalities hold: 
\begin{enumerate}
\item When $t \in (0, 1/8)$, we have 
$h^0(F)/r \leq \frac{16(t+3)}{49}$. 

\item When $t \in [1/8, 1/6)$, we have 
$h^0(F)/r \leq \max\left\{
	\frac{12t+24}{25}, 
	 \frac{85t}{246}+\frac{481}{492}
	\right\}$. 

\item When $t \in [1/6, 1/4)$, we have 
$h^0(F)/r \leq \max\left\{
	\frac{8t+8}{9}, 
	 \frac{17t}{38}+\frac{147}{152}
	\right\}$. 

\item When $t \in [1/4, 1/2]$, we have 
$h^0(F)/r \leq \max\left\{
	4t, 
	 \frac{63t}{82}+\frac{153}{164}
	\right\}$. 
	
\item When $t \in [3/2, 15/8]$, we have 
$h^0(F)/r \leq 4t$. 

\item When $t \in (15/8, \frac{\sqrt{23}-1}{2}]$, 
we have 
$h^0(F)/r \leq 
\frac{133t-114}{18}$. 

\item When $t \in [\frac{\sqrt{23}-1}{2}, 31/16]$, 
we have 
$h^0(F)/r \leq 
\frac{236}{49}t-\frac{148}{21}$. 

\item When $t \in [31/16, 2)$, 
we have 
$h^0(F)/r \leq 16t-23$. 
\end{enumerate}
\end{prop}
\begin{proof}
As in the proof of Proposition \ref{prop:clifford}, 
the problem is reduced to computing 
$\Omega(\overrightarrow{OP_1})+\Omega(\overrightarrow{P_1P})$ 
for appropriate candidate points $P_1$ 
in the triangle $OPQ$. 
Here the points $P, Q$ are defined as before, namely, 
$P:=(\ch_2(\iota_*F),H\ch_1(\iota_*F))=(16(t-4)r, 16r)$, 
and $Q=(x_q, y_q)$ is the point 
such that $x_q/y_q$ is the upper bound for 
$\nu^+_{BN}(\iota_{*}F)$, 
and $(x_p-x_q)/(y_p-y_q)$ is the lower bound for 
$\nu^-_{BN}(\iota_{*}F)$, given in Lemma \ref{lem:bnslope2}. 

(0) First assume that 
$t \in \left[0, \frac{5-\sqrt{23}}{2} \right]$. 
In this case, the sheaf $\iota_*F$ is BN stable 
with BN slope 
$\nu_{BN}(\iota_*F)=t-4 \in (-4, -3)$. 
Hence we have 
\[
h^0(F)/r \leq 
\frac{16}{7}+\frac{16(t-4)}{49}
=\frac{16(t+3)}{49} 
\]
by Lemma \ref{lem:sectionbn}. 

(1) Assume $t \in \left[\frac{5-\sqrt{23}}{2}, \frac{1}{8} \right)$. 
By Lemma \ref{lem:bnslope2}, we have 
\begin{itemize}
\item The slope of $\overrightarrow{OP}$ 
is $\frac{1}{t-4} \in (-1/3, -1/4)$, 

\item The slope of $\overrightarrow{OQ}$ 
is $\frac{2t}{2t-1} \in (-1/3, -1/4)$, 

\item The slope of $\overrightarrow{PQ}$ 
is $\frac{2(t-8)}{-7(2t-9)} \in (-1/3, -1/4)$. 
\end{itemize}
Hence we can assume $P_1=Q$, and get 
\[
h^0(F)/r \leq 
\left(
	\Omega(\overrightarrow{OQ})+\Omega(\overrightarrow{QP})
	\right)/r 
=\frac{16}{7}+\frac{16(t-4)}{49}
=\frac{16(t+3)}{49}. 
\]

(2) Assume $t \in [1/8, 1/6)$. 
Then the slope bounds on $\nu_{BN}(\iota_*F)$ 
are the same as the case (1), 
but the slope of $\overrightarrow{OQ}$ 
is in the interval $(-1/2, -1/3]$. 
Hence we may take $P_1$ to be 
$Q$ or $A$, where 
$A$ is the point on the line $PQ$ with 
slope $-1/3$. 
The coordinates are given as 
\[
Q=((2t-1)r, 2tr), \quad 
A=\left(
	\frac{48(2t^2-10t+1)}{-8t+15}r, 
	\frac{16(2t^2-10t+1)}{8t-15}r 
	\right). 
\]
Hence we have 
\begin{align*}
&\Omega(\overrightarrow{OQ})+\Omega(\overrightarrow{QP})
=\frac{y_q}{5}+\frac{x_q}{25}
	+\frac{y_p-y_q}{7}+\frac{x_p-x_q}{49} 
=\frac{12t+24}{25}r, \\
&\Omega(\overrightarrow{OA})+\Omega(\overrightarrow{AP})
=\frac{y_a}{12}
	+\frac{y_p-y_a}{7}+\frac{x_p-x_a}{49} 
=\frac{1}{588}y_a+\frac{16(t+3)}{49}r. 
\end{align*}
As a function on $t \in [1/8, 1/6)$, we have 
\[
y_a(t) \leq \frac{458t-47}{41}r, 
\]
and hence 
\[
\Omega(\overrightarrow{OA})+\Omega(\overrightarrow{AP}) 
\leq \frac{85t}{246}r+\frac{481}{492}r. 
\]
We conclude that 
\[
h^0(F)/r \leq \max\left\{
	\frac{12t+24}{25}, 
	 \frac{85t}{246}+\frac{481}{492}
	\right\}. 
\]

(3) Assume that $t \in [1/6, 1/4)$. 
Then the slopes of 
$\overrightarrow{OP}, \overrightarrow{OQ}, \overrightarrow{QP}$ 
are same as in the case (1), 
but the slope of $\overrightarrow{OQ} \in (-1, -1/2]$ instead. 
There are three possibilities of the point $P_1$, 
namely, $Q, A$, and $B$. 
Here, $A, B$ are the points 
on the line $PQ$ with slope 
$-1/3, -1/2$, respectively. 
The coordinates of $Q, A$ 
are same as in the case (2), and 
\[
B=\left(
	\frac{32(2t^2-10t+1)}{-10t+31}r, 
	\frac{16(2t^2-10t+1)}{10t-31}r
	\right). 
\]
We therefore get 
\begin{align*}
&\Omega(\overrightarrow{OQ})+\Omega(\overrightarrow{QP})
=\frac{y_q}{3}+\frac{x_q}{9}
	+\frac{y_p-y_q}{7}+\frac{x_p-x_q}{49} 
=\frac{8(t+1)}{9}r, \\
&\Omega(\overrightarrow{OA})+\Omega(\overrightarrow{AP})
=\frac{y_a}{12}
	+\frac{y_p-y_a}{7}+\frac{x_p-x_a}{49} 
=\frac{1}{588}y_a+\frac{16(t+3)}{49}r, \\
&\Omega(\overrightarrow{OB})+\Omega(\overrightarrow{BP})
=\frac{y_b}{8}
	+\frac{y_p-y_b}{7}+\frac{x_p-x_b}{49}
=\frac{9}{392}y_b+\frac{16(t+3)}{49}r. 
\end{align*}
We can see that 
$\Omega(\overrightarrow{OA})+\Omega(\overrightarrow{AP}) \leq 
\Omega(\overrightarrow{OB})+\Omega(\overrightarrow{BP})$. 
Furthermore, as a function on $t \in [1/6, 1/4)$, we have 
$y_b(t) \leq 
\frac{100}{19}tr-\frac{31}{57}r$, 
and hence we have 
\[
\Omega(\overrightarrow{OB})+\Omega(\overrightarrow{BP}) 
\leq \frac{17}{38}tr+\frac{147}{152}r. 
\]

(4) Assume that $t \in [1/4, 1/2]$. 
In this case, the slope of $\overrightarrow{OQ}$ 
is bigger than or equal to $-1$. 
Hence we may take $P_1=Q, A, B$, or $C$, 
where $A, B$ are defined as in (3), and 
$C$ is the point on the line $PQ$ 
with slope $-1$. 
We have 
\[
C=\left(
	-\frac{16(2t^2-10t+1)}{12t-47}r, 
	\frac{16(2t^2-10t+1)}{12t-47}r
	\right). 
\]
We have 
\begin{align*}
&\Omega(\overrightarrow{OQ})+\Omega(\overrightarrow{QP})
=y_q+x_q
	+\frac{y_p-y_q}{7}+\frac{x_p-x_q}{49} 
=4tr, \\
&\Omega(\overrightarrow{OC})+\Omega(\overrightarrow{CP})
=\frac{y_c}{4}
	+\frac{y_p-y_c}{7}+\frac{x_p-x_c}{49} 
=\frac{25}{196}y_c+\frac{16(t+3)}{49}r, 
\end{align*}
and $\Omega(\overrightarrow{OA})+\Omega(\overrightarrow{AP})$, 
$\Omega(\overrightarrow{OB})+\Omega(\overrightarrow{CB})$ 
are as in (3). 
Hence we can show that 
\[
\Omega(\overrightarrow{OA})+\Omega(\overrightarrow{AP}), 
\Omega(\overrightarrow{OB})+\Omega(\overrightarrow{BP}) \leq 
\Omega(\overrightarrow{OC})+\Omega(\overrightarrow{CP}). 
\]

On the other hand, as a function on $t \in [1/4, 1/2]$, we have 
$y_c(t) \leq \frac{142}{41}tr-\frac{15}{41}r$, and so 
\[
\Omega(\overrightarrow{OC})+\Omega(\overrightarrow{CP}) 
\leq \frac{63}{82}tr+\frac{153}{164}r. 
\]

(5) Assume that $t \in [3/2, 15/8]$. 
In this case, we have 
\begin{itemize}
\item The slope of $\overrightarrow{OP}$ 
is $\frac{1}{t-4} \in (-1/2, -1/3)$, 

\item The slope of $\overrightarrow{OQ}$ 
is $\frac{2t}{2t-1}>0$, 

\item The slope of $\overrightarrow{PQ}$ 
is $\frac{2(t-8)}{-7(2t-9)} \in [-1/3, -1/4)$. 
\end{itemize}
Hence we may choose $P_1$ as 
$Q, B, C, D, E, F$, where 
$B, C$ are defined as in (4), 
$D$ is the point on the line $OQ$ 
such that the slope of $DP$ is $-1/3$, 
and $E$ (resp. $F$) are the intersection points 
of the lines $DP$ and $OB$ (resp. $OC$). 
Hence the computations of 
$\Omega(\overrightarrow{OB})+\Omega(\overrightarrow{BP}), 
\Omega(\overrightarrow{OC})+\Omega(\overrightarrow{CP})$, 
and $\Omega(\overrightarrow{OQ})+\Omega(\overrightarrow{QP})$ 
are exactly the same as in (4), and we can see that 
\[
\Omega(\overrightarrow{OB})+\Omega(\overrightarrow{BP}), 
\Omega(\overrightarrow{OC})+\Omega(\overrightarrow{CP}) \leq 
\Omega(\overrightarrow{OQ})+\Omega(\overrightarrow{QP}). 
\]

For $D, E$ and $F$, the coordinates are given as 
\begin{align*}
&D=\left(
	\frac{16(2t-1)(t-1)}{8t-1}r, 
	\frac{32t(t-1)}{8t-1}r 
	\right), \quad 
E=(-(32t-32)r, (16t-16)r), \\
&F=(-(8t-8)r, (8t-8)r). 
\end{align*}
We get 
\begin{align*}
\Omega(\overrightarrow{OD})+\Omega(\overrightarrow{DP})
&=y_d+x_d+\frac{y_p-y_d}{12}, \\
\Omega(\overrightarrow{OE})+\Omega(\overrightarrow{EP})
&=\frac{y_e}{8}+\frac{y_p-y_e}{12}=\frac{2}{3}(t+1)r \\
\Omega(\overrightarrow{OF})+\Omega(\overrightarrow{FP})
&=\frac{y_f}{4}+\frac{y_p-y_f}{12}=\frac{4}{3}tr. 
\end{align*}
We can also see that 
$\Omega(\overrightarrow{OD})+\Omega(\overrightarrow{DP}) 
\leq \Omega(\overrightarrow{OQ})+\Omega(\overrightarrow{QP})=4tr$, 
and hence we conclude that 
\[
h^0(F)/r \leq 4t. 
\]

(6) Assume $t \in \left(15/8, \frac{\sqrt{23}-1}{2} \right]$. 
The only difference with (5) is that 
the slope of $\overrightarrow{QP}$ 
is equal to $-1/3$ in the present case. 
Hence we may choose $P_1$  to be 
$Q, E$, or $F$ appeared in (5). 
It is easy to see that 
\begin{align*}
&\Omega(\overrightarrow{OE})+\Omega(\overrightarrow{EP})=
\Omega(\overrightarrow{OF})+\Omega(\overrightarrow{FP}) 
\leq \Omega(\overrightarrow{OQ})+\Omega(\overrightarrow{QP}), \\
&\Omega(\overrightarrow{OQ})+\Omega(\overrightarrow{QP})
=\frac{4(46t^2-50t+11)}{3(8t-1)}r 
\leq \frac{133t-114}{18}r. 
\end{align*}

(7) Assume that 
$t \in \left[\frac{\sqrt{23}-1}{2}, \frac{31}{16} \right]$. 
Then we have 
\begin{itemize}
\item The slope of $\overrightarrow{OP}$ 
is $\frac{1}{t-4} \in (-1/2, -1/3)$, 

\item The slope of $\overrightarrow{OQ}$ 
is $\frac{12t-11t}{13t-16} >0$, 

\item The slope of $\overrightarrow{PQ}$ 
is $-1/3$. 
\end{itemize}
Hence we may choose $P_1$ to be 
$Q, E$, or $F$, where 
the points $E, F$ are defined as in (5). 
We have 
\[
Q=\left(
	\frac{16(13t-16)}{49}r, 
	\frac{16(12t-11)}{49}r
	\right) 
\]
and hence 
\[
\Omega(\overrightarrow{OQ})+\Omega(\overrightarrow{QP})
=y_q+x_q
	+\frac{y_p-y_q}{12}
=\frac{236}{49}tr-\frac{148}{21}r. 
\]
On the other hand, from the computations in (5), 
we see that 
\[
\Omega(\overrightarrow{OE})+\Omega(\overrightarrow{EP})
=\Omega(\overrightarrow{OF})+\Omega(\overrightarrow{FP})
=\frac{4}{3}tr 
\leq \Omega(\overrightarrow{OQ})+\Omega(\overrightarrow{QP}). 
\]
We conclude that 
\[
h^0(F)/r \leq 
\frac{236}{49}t-\frac{148}{21}. 
\]

(8) Assume $t \in [31/16, 2)$. 
Then we have 
\begin{itemize}
\item The slope of $\overrightarrow{OP}$ 
is $\frac{1}{t-4} \in (-1/2, -1/3)$, 

\item The slope of $\overrightarrow{OQ}$ 
is $\frac{1}{4t-7} >0$, 

\item The slope of $\overrightarrow{PQ}$ 
is $-1/3$. 
\end{itemize}
Hence we may choose $P_1$ to be 
$Q, E$, or $F$, where 
the points $E, F$ are defined as in (5). 
We have 
\[
Q=\left(
	4(4t-7)r, 
	4r
	\right) 
\]
and hence 
\[
\Omega(\overrightarrow{OQ})+\Omega(\overrightarrow{QP})
=y_q+x_q
	+\frac{y_p-y_q}{12}
=(16t-23)r. 
\]
As in (7), we see that 
\[
\Omega(\overrightarrow{OE})+\Omega(\overrightarrow{EP})
=\Omega(\overrightarrow{OF})+\Omega(\overrightarrow{FP})
=\frac{4}{3}tr 
\leq \Omega(\overrightarrow{OQ})+\Omega(\overrightarrow{QP}), 
\]
and we can conclude that 
\[
h^0(F)/r \leq 
16t-23. 
\]

\end{proof}

\subsection{Strong (classical) BG inequality}
Using Proposition \ref{prop:clifford2}, 
we get the following (classical) BG type inequality on 
a double cover CY3 $X$. 
\begin{thm} \label{thm:BGonT2} 
Let $X$ be a double cover CY3. 
Let $F \in D^b(X)$ be a $\nu_{\alpha, 0}$-semistable object 
for some $\alpha >0$, with $\mu_H(E) \in (-1, 1)$. 
Then we have the inequality 
\[
\frac{H\ch_2(F)}{H^3\ch_0(F)} 
\leq \Xi \left(\left|\frac{H^2\ch_1(F)}{H^3\ch_0(F)} \right|\right).  
\]
Here the function $\Xi$ is defined 
as in Theorem \ref{thm:strongBG}, i.e., 
\[
\Xi(t)=\left\{ \begin{array}{ll}
t^2-t & (t \in [0, 1/4]) \\ 
3t/4-3/8 & (t \in [1/4, 1/2]) \\
t/4-1/8 & (t \in [1/2, 3/4]) \\
t^2-1/2 & (t \in [3/4, 1]). 
\end{array} \right.
\]
\end{thm}

\begin{proof}
As in the proof of Theorem \ref{thm:strongBG}, 
the problem is reduced to proving the same statement 
for tilt-semistable objects on $T$. 
Assume that there exists a tilt semistable object $F$ on $T$ 
violating the inequality in the statement. 
As before, we may assume that 
$\mu(F) \in (0, 1/2]$ and $F|_C$ is slope semistable. 
Then we have (cf.  equation (\ref{eq:rr})) 
\begin{equation} \label{eq:rrT2}
\begin{aligned} 
\ch_{2}(F)-H_{T}\ch_{1}(F)+10\ch_{0}(F) 
&=\chi(F) \\
&\leq h^0(F|_C)+h^0(F^{\vee}(2H_{T})|_{C}), 
\end{aligned}
\end{equation}
and 
\begin{align*}
&\ch(F|_C)=(\ch_0(F), 4H\ch_1(F)) ,\\ 
&\ch(F^\vee(2H)|_C)=(\ch_0(F), 4(2H^2\ch_0(F)-H\ch_1(F))). 
\end{align*}
By applying Proposition \ref{prop:clifford2} 
to the right hand side of the inequality (\ref{eq:rrT2}), 
we have 

\[ 
\frac{\ch_{2}(F)}{H^2\ch_{0}(F)} 
\leq 
\begin{cases}
-\frac{143}{49}\mu_H(F)-\frac{23}{98} 
& (\mu_H(F) \in (0, 1/16]) \\
-\frac{6}{49}\mu_H(F)-\frac{1081}{588} 
& (\mu_{H}(F) \in [1/16, \frac{5-\sqrt{23}}{2}]) \\ 
-\frac{2701}{3528}\mu_H(F)-\frac{127}{882}
& (\mu_{H}(F) \in [ \frac{5-\sqrt{23}}{2}, 1/8)) \\
\frac{85}{984}\mu_H(F)-\frac{503}{1968} 
& (\mu_{H}(F) \in [1/8, 217/1654]) \\ 
\frac{3}{25}\mu_H(F)-\frac{13}{50} 
& (\mu_{H}(F) \in [217/1654, 1/6)) \\ 
\frac{17}{152}\mu_H(F)-\frac{157}{608} 
& (\mu_{H}(F) \in [1/6, 107/604] \\
\frac{2}{9}\mu_H(F)-\frac{5}{18} 
& (\mu_{H}(F) \in [107/604, 1/4)) \\
\frac{63}{328}\mu_H(F)-\frac{175}{656}
& (\mu_H(F) \in [1/4, 153/530]) \\
\mu_H(F)-\frac{1}{2}
& (\mu_H(F) \in [153/530, 1/2]), 
\end{cases}
\]
which is a cotradiction. 
\end{proof}

\section{BG type inequality conjecture} \label{section:BGconj}
In this section, we will prove that 
the strong BG inequality in Theorem \ref{thm:mainintro} 
implies Theorem \ref{thm:main2intro}. 
We work in the following general set up. 
Let $X$ be a smooth projective Calabi-Yau threefold, 
$H$ a nef and big divisor on $X$. 
Let us put $d:=H^3$, $e:=H.\td_{X, 2}$. 
We define the positive real number 
$\delta_X=\delta_X(H)$ as follows: 
\[
\delta_X:=\max\left\{
	\frac{4}{d}, 
	\frac{e}{d}, 
	\frac{26}{3d}-\frac{e}{d}-\frac{1}{3}, 
	\frac{16-3e}{3d}
	\right\}. 
%
\]
Note that we always have 
$\frac{57-7e}{13d} < \delta_X$. 

It is easy to compute the number $\delta_X$ in the following cases: 
\begin{itemize}
\item When $X$ is a triple cover CY3, 
we have $\delta_X=25/18$. 

\item When $X$ is a double cover CY3, 
we have $\delta_X=13/6$.  
\end{itemize}

\begin{rmk}
As pointed out by the referee, there exist some inequalities between $d$ and $e$, 
see \cite[Proposition 2.2]{kw14}. 
For example, when $H$ is very ample, then the inequality 
$3e \leq d+9$ holds 
by \cite[Proposition 2.2 (1)]{kw14}. This implies 
\[
\frac{26}{3d}-\frac{e}{d}-\frac{1}{3} \geq \frac{17}{3d} > \frac{4}{d}
\]
and hence the definition of $\delta_X$ simplifies. 
\end{rmk}

We put the following assumption: 
\begin{assump} \label{assump*}
Every object $E \in D^b(X)$, which is $\nu_{0, \alpha}$-semistable 
for some $\alpha>0$, satisfies the strong BG inequality 
(\ref{eq:strongBGintro}). 
\end{assump}

\begin{prop} \label{prop:ch2-ch3}
Let $X$ be a smooth projective Calabi-Yau threefold, 
H a nef and big divisor on $X$. 
For a real number $\delta \geq \delta_X$, 
define a $1$-cycle $\Gamma$ as 
$\Gamma:=\delta H^2-\td_{X, 2}$. 
Assume that Assumption \ref{assump*} holds.
Then every BN stable object $E \in D^b(X)$ 
with $\nu_{BN}(E) \in [0, 1/2]$ satisfies the inequality 
\begin{align*}
Q^\Gamma(E):=Q^\Gamma_{0, 0}(E)&=
	2(H\ch_{2}(E))(2H\ch_2(E)-3\Gamma H\ch_0(E)) \\
	&\quad -6(H^2\ch_{1}(E))(\ch_{3}(E)-\Gamma\ch_1(E)) \geq 0. 
\end{align*}
\end{prop}
\begin{proof}
First assume that $\nu_{BN}(E) \in (0, 1/2]$. 
Let us consider the universal extension
\[
E \to \widetilde{E} \to \Hom(\mcO_{X}, E) \otimes \mcO_X[1]. 
\]
By \cite[Lemma 2.12]{li19b}, 
$\widetilde{E}$ is BN semistable with 
$\nu_{BN}(\widetilde{E})=\nu_{BN}(E)$. 
By Assumption \ref{assump*}, 
we can see that 
$\mu_H(\widetilde{E}) \notin (-1/4, 0]$. 
Indeed, if otherwise, we have 
\[
\frac{H\ch_2(\widetilde{E})}{H^3\ch_0(\widetilde{E})} 
\leq \mu_H(\widetilde{E})^2
	+\mu_H(\widetilde{E})
< \frac{1}{2}\mu_H(\widetilde{E}). 
\]
Dividing both sides by 
$\mu_H(\widetilde{E}) (<0)$, 
we get $\nu_{BN}(\widetilde{E}) > 1/2$, 
a contradiction. 
When $\mu_H(\widetilde{E}) \in [-1/2, -1/4]$, 
using Assumption \ref{assump*} 
to the object $\widetilde{E}$, 
we have 
\begin{equation} \label{eq:tildeBG}
\frac{H\ch_2(\widetilde{E})}{H^3\ch_0(\widetilde{E})} 
\leq -\frac{3}{4}\mu_H(\widetilde{E})
	-\frac{3}{8}. 
\end{equation}
Note that we have 
$\ch_0(\widetilde{E})=\ch_0(E)-\hom(\mcO_X, E)$ and 
$\ch_i(\widetilde{E})=\ch_i(E)$ for $i=1, 2$. 
Note also that we have 
$H^2\ch_1(E) \geq 0$ since $E \in \Coh^0(X)$. 
Together with the assumption $\mu_H(\widetilde{E}) <0$, 
we have $\ch_0(\widetilde{E}) <0$. 
From these observations, 
the inequality (\ref{eq:tildeBG}) 
is equivalent to the inequality 
\begin{equation} \label{eq:lowerb}
\hom(\mcO_X, E) \leq 
\frac{8}{3d}H\ch_2(E)+\frac{2}{d}H^2\ch_1(E)+\ch_0(E). 
\end{equation}

On the other hand, the inequality 
$\mu_H(\widetilde{E}) < -1/2$ 
is equivalent to 
\[
\hom(\mcO_X, E) \leq 
\frac{2}{d}H^2\ch_1(E)+\ch_0(E), 
\]
which is stronger than (\ref{eq:lowerb}) 
since we have $H\ch_2(E) >0$ 
by our assumption $\nu_{BN}(E) >0$. 
If $\mu_{H}(\widetilde{E})>0$, 
the same inequality (\ref{eq:lowerb}) 
obviously holds 
since $\ch_{0}(E)-\hom(\mcO_{X}, E)=\ch_{0}(\widetilde{E})>0$ 
and $H^2\ch_{1}(E) >0$. 
Hence the inequality (\ref{eq:lowerb}) always holds. 

On the other hand, by using the BN stability of 
$E$ and $\mcO_{X}[1]$, we have 
\begin{equation} \label{eq:upperb}
\hom(\mcO_{X}, E) \geq \chi(E) 
=\ch_{3}(E)+\td_{X, 2}\ch_1(E). 
\end{equation}

Combining the inequalities (\ref{eq:lowerb}) 
and (\ref{eq:upperb}), we get 
\[
\ch_{3}(E)+\td_{X, 2}\ch_1(E) \leq 
\frac{1}{d}H^3\ch_{0}(E)+\frac{2}{d}H^2\ch_{1}(E)+\frac{8}{3d}H\ch_2(E), 
\]
and hence 
\begin{equation} \label{eq:q}
\begin{aligned}
Q^{\Gamma}(E) &\geq 
4(H\ch_{2}(E))^2
	-6\cdot\frac{\delta d-e}{d}H\ch_2(E)H^3\ch_0(E) 
	+6\delta\left( H^2\ch_1(E) \right)^2 \\
&\quad -6H^2\ch_1(E)\left( 
	\frac{1}{d}H^3\ch_0(E)+\frac{2}{d}H^2\ch_1(E)+\frac{8}{3d}H\ch_2(E)
	\right) \\
&=4b^2-\frac{16}{d}ab+6\left(\delta-\frac{2}{d}\right)a^2
	-6\left(\delta-\frac{e}{d}\right)rb-\frac{6}{d}ra. 
\end{aligned}
\end{equation}
Here we put 
$(r, a, b):=(H^3\ch_0(E), H^2\ch_1(E), \ch_2(E))$, 
to simplify the notation. 
Note that we have $a \geq 0$, 
since $E \in \Coh^0(X)$. 
Moreover, we also have $b \geq 0$ 
from the assumption $\nu_{BN}(E) \geq 0$. 
Note also that by definition of $\delta_X$, 
we have $\delta-2/d, \delta-e/d \geq 0$. 

By Assumption \ref{assump*}, 
we know that $\mu_H(E) \notin [0, 1/2]$. 
When $\mu_H(E) \notin [1/2, 1]$, 
we have $r<a$.  
Together with the inequality (\ref{eq:q}), we have 
\begin{align*}
Q^\Gamma(E) & \geq 
4b^2-\frac{16}{d}ab+6\left(\delta-\frac{2}{d}\right)a^2
	-6\left(\delta-\frac{e}{d}\right)ab-\frac{6}{d}a^2 \\
&=A_1a^2-B_1ab+4b^2, 
\end{align*}
where we put 
$A_1:=6(\delta-3/d), 
B_1:=6(\delta-e/d)+16/d > 0$. 
We can further compute as 
\begin{equation} \label{eq:mucase1}
A_1a^2-B_1ab+4b^2
=(a-2b)\left(A_1a+(2A_1-B_1)b \right)
+(4A_1-2B_1+4)b^2. 
\end{equation}
By the assumption $0< \nu_{BN}(E) \leq 1/2$ 
we have $0 < 2b \leq a$. 
Moreover, we also have 
$\delta \geq \delta_X \geq 
\frac{26}{3d}-\frac{e}{d}-\frac{1}{3}$ 
by definition. 
From these we can conclude that 
the right hand side of the equality (\ref{eq:mucase1}) 
is non-negative, 
and hence we have $Q^\Gamma(E) \geq 0$ as required. 

When $\mu_H(E) \in [1/2, 3/4]$, 
by Assumption \ref{assump*}, 
we have 
\[
-r \geq -2a+8b. 
\]
Combining with the inequality (\ref{eq:q}), we have 
\begin{align*}
Q^\Gamma(E) & \geq 
4b^2-\frac{16}{d}ab+6\left(\delta-\frac{2}{d} \right)a^2
+6\left(\delta-\frac{e}{d} \right)(-2a+8b)b
+\frac{6}{d}(-2a+8b)a \\
&=\left(4+48\left(\delta-\frac{e}{d} \right) \right)b^2
+\left(\frac{32}{d}-12\left(\delta-\frac{e}{d} \right) \right)ab 
+6\left(\delta-\frac{e}{d} \right)a^2 \\
&=:C_2b^2-B_2ab+A_2a^2 \\
&=(a-2b)(A_2a+(2A_2-B_2)b)+(4A_2-2B_2+C_2)b^2, 
\end{align*}
where the real numbers $A_2, B_2, C_2$ 
are defined so that the second equality holds. 
Since $\delta \geq 4/d, 3/d$, 
we can see that 
$4A_2-2B_2+C_2 \geq0$. 
Moreover, using the inequalities 
$0 < 2b \leq a$ and 
$\delta \geq -e/d+16/3d$, 
we also obtain 
$A_2a+(2A_2-B_2)b \geq 0$. 
Hence we have $Q^\Gamma(E) \geq 0$. 

Next consider the case when 
$\mu(E) \in [3/4, 1]$. 
By Assumption \ref{assump*}, 
we have 
$b/r \leq 7a/4r-5/4$, equivalently, 
\[
-r \geq -\frac{7}{5}a+\frac{4}{5}b. 
\]
Together with the inequality (\ref{eq:q}), we have 
\begin{align*}
5Q^\Gamma(E) &\geq 
20b^2-\frac{80}{d}ab+30\left(\delta-\frac{2}{d}\right)a^2 \\
	&\quad +6\left(\delta-\frac{e}{d}\right)b\left(-7a+4b \right)
	+\frac{6}{d}a\left(-7a+4b \right) \\
&=\left(20+24\left(\delta-\frac{e}{d} \right) \right)b^2
-\left(42\left(\delta-\frac{e}{d} \right)-\frac{66}{d} \right)ab
+\left(30\delta-\frac{102}{d} \right)a^2 \\
&=:C_3b^2-B_3ab+A_3a^2 \\
&=(a-2b)(A_3a+(2A_3-B_3)b)+(4A_3-2B_3+C_3)b^2. 
\end{align*}
Using the inequalities $a \geq 2b$ 
and 
$\delta \geq -e/d+46/10d-1/3, (57-7e)/13d$, 
we can show that $Q^\Gamma(E) \geq 0$. 

The remaining case is when $\nu_{BN}(E)=0$. 
The issue is that we do not know 
whether $\widetilde{E}$ is BN semistable or not. 
If it is $\nu_{\alpha, 0}$-semistable for some $\alpha>0$, 
as in the case of $\nu_{BN}(E)>0$, 
we have the inequality 
\begin{equation} \label{eq:upbd2}
\hom(\mcO_X, E) \leq 
\frac{2}{d}H^2\ch_1(E)+\ch_0(E). 
\end{equation}

Assume that $\widetilde{E}$ is $\nu_{\alpha, 0}$-unstable 
for all $\alpha>0$. 
Then by the proof of \cite[Proposition 3.3]{li19b}, 
for each $0 < \delta \ll 1$, there exists 
$\alpha_i>0$ and a filtration of $E$ 
such that each factor 
$E_i$ is $\nu_{\alpha_i, 0}$-semistable 
with $\nu_{BN}(E_i)<\delta$. 
By Assumption \ref{assump*}, 
we must have 
\[
\mu_H(E_i) \notin \left[
-\frac{3}{8\delta+6}, 0
\right]. 
\]
Taking a limit $\delta \to +0$, 
we get 
$\mu(\widetilde{E}) \notin [-1/2, 0]$, hence 
the inequality (\ref{eq:upbd2}) holds. 
Furthermore, by using the derived dual 
(cf. proof of \cite[Proposition 3.3]{li19b}), 
we also have 
\[
\hom(\mcO_X, E[2]) \leq \frac{2}{3}H^2\ch_1(E)-\ch_0(E).
\]
Hence we get 
\begin{align*}
\ch_3(E)+\td_{X, 2}\ch_1(E)
=\chi(E) 
&\leq \hom(\mcO_X, E)+\hom(\mcO_X, E[2]) \\
&\leq \frac{4}{d}H^2\ch_1(E), 
\end{align*}
from which we deduce 
$Q^\Gamma(E) \geq 0$, 
as we assume $\delta \geq 4/d$. 
\end{proof}

\begin{cor} \label{cor:bgconj}
Let $X$ be a triple (resp. double) CY3. 
We put $\gamma:=2/9$ (resp. $1/3$) 
and $\Gamma:=\gamma H^2$. 
Let $E$ be a BN stable object on $X$ 
with $\nu_{BN}(E) \in [0, 1/2]$. 
Then we have 
\begin{align*}
Q^\Gamma(E):=Q^\Gamma_{0, 0}(E)&=
	2(H\ch_{2}(E))(2H\ch_2(E)-3\Gamma H\ch_0(E)) \\
	&\quad -6(H^2\ch_{1}(E))(\ch_{3}(E)-\Gamma\ch_1(E)) \geq 0. 
\end{align*}
\end{cor}
\begin{proof}
By Theorems \ref{thm:strongBG}, \ref{thm:BGonT2}, 
a triple/double CY3 satisfies 
Assumption \ref{assump*}. 
Furthermore, we can take the $1$-cycle $\Gamma$ to be 
$\Gamma:=\delta_XH^2-\td_{X, 2}=\gamma H^2$. 
\end{proof}

\section{Construction of Bridgeland stability conditions} \label{section:conststab}
The goal of this section is 
to prove Theorem \ref{thm:main3intro} 
in the introduction. 
First let us recall the definition of 
Bridgeland stability condition. 

\begin{defin}[\cite{bri07}] \label{defin:stab}
Let $\mathcal{D}$ be a triangulated category. 
Fix a lattice $\Lambda$ of finite rank 
and a group homomorphism 
$\cl \colon K(\mathcal{D}) \to \Lambda$. 

A {\it stability condition} on $\mathcal{D}$ 
(with respect to $(\Lambda, \cl)$) is a pair 
$(Z, \mcA)$ consisting of 
a group homomorphism 
$Z \colon \Lambda \to \bC$ 
and the heart of a bounded t-structure 
$\mcA \subset \mathcal{D}$ 
satisfying the following axioms. 
\begin{enumerate}
\item We have 
$Z \circ \cl \left(\mcA \setminus \{0\} \right) 
\subset \mathbb{H} \cup \bR_{<0}$, 
where $\mathbb{H}$ is the upper half plane. 

\item Every non-zero object 
in the heart $\mcA$ 
has a Harder-Narasimhan filtration 
with respect to $\mu_Z$-stability. 
Here we define a $Z$-slope function $\mu_Z$ 
as 
\[
\mu_Z:=-\frac{\Re Z}{\Im Z} 
\colon \mcA 
\to \bR \cup \{+\infty\}, 
\]
and define $\mu_Z$-stability 
on the abelian category $\mcA$ 
in a usual way. 

\item There exists a quadratic form $q$ on $\Lambda$ 
satisfying the following conditions. 
\begin{itemize}
\item $q$ is negative definite on the kernel of $Z$, 

\item For every $\mu_Z$-semistable object $E \in \mcA$, 
we have $q(\cl(E)) \geq 0$. 
\end{itemize}
\end{enumerate}

The group homomorphism $Z$ is called a {\it central charge}, 
and the axiom (3) is called the {\it support property}. 
\end{defin}

Let $\Stab_\Lambda(\mathcal{D})$ be a set 
of stability conditions on $\mathcal{D}$ 
with respect to $(\Lambda, \cl)$. 
Then the set $\Stab_\Lambda(\mathcal{D})$ has 
a structure of a complex manifold \cite{bri07}. 
Moreover, there is an action of the group 
$\widetilde{\GL^+}(2, \bR)$ 
on $\Stab_\Lambda(\mathcal{D})$, 
where $\widetilde{\GL^+}(2, \bR)$ 
is the universal covering of 
the group 
\[
\GL^+(2, \bR):=\{g \in \GL(2, \bR) : \det(g)>0
\}. 
\]

Let us consider the case when 
$\mathcal{D}=D^b(X)$, 
where $X$ is a double/triple cover CY3. 
In this case, we fix a lattice $\Lambda$ to be 
the image of the morphism 
\[
\cl:=\left(H^3\ch_0, H^2\ch_1, H\ch_2, \ch_3 \right) 
\colon K(X) \to H^{2*}(X, \mathbb{Q}). 
\]
We simply denote as 
$\Stab(X):=\Stab_{\Lambda}(D^b(X))$. 
Following \cite{bms16, bmt14a}, 
we explain an explicit construction 
of stability conditions on $D^b(X)$. 
Let us recall several notions 
from \cite{bms16}. 
Fix real numbers 
$\alpha, \beta \in \bR$ 
with $\alpha>0$. 

The heart corresponding to a stability condition 
is constructed as a tilt of $\Coh^\beta(X)$. 
Let us define a slope function 
$\nu'_{\beta, \alpha}$
on $\Coh^\beta(X)$ as 
\[
\nu'_{\beta, \alpha}:=
\frac{H\ch_2^\beta-\frac{1}{2}\alpha^2H^3\ch_0^\beta}{H^2\ch_1^\beta} 
\colon \Coh^\beta(X) \to \bR \cup \{+\infty\}. 
\]
Compared with the function 
$\nu_{\beta, \alpha}$ 
defined in Section \ref{section:preliminaries}, 
we have 
\begin{equation} \label{eq:compare}
\nu'_{\beta, \alpha}=\nu_{\beta, \frac{1}{2}(\beta^2+\alpha^2)}-\beta. 
\end{equation}

We define full subcategories 
$\mcT'_{\beta, \alpha}, \mcF'_{\beta, \alpha}$ 
of $\Coh^\beta(X)$ as 
\begin{align*}
&\mcT'_{\beta, \alpha}:=\left\langle
	T \in \Coh^\beta(X) : 
	T \mbox{ is } \nu'_{\beta, \alpha} \mbox{-semistable with } 
	\nu'_{\beta, \alpha}(E)>0
	\right\rangle, \\
&\mcF'_{\beta, \alpha}:=\left\langle
	F \in \Coh^\beta(X) : 
	F \mbox{ is } \nu'_{\beta, \alpha} \mbox{-semistable with } 
	\nu'_{\beta, \alpha}(E) \leq 0
	\right\rangle. 	
\end{align*}

Here, $\nu'_{\beta, \alpha}$-stability 
is same as 
$\nu_{\beta, \frac{1}{2}(\beta^2+\alpha^2)}$-stability, 
and $\langle-\rangle$ denotes 
the extension closure in the abelian category 
$\Coh^\beta(X)$. 
We now define the {\it double-tilted heart} as 
\[
\mcA^{\beta, \alpha}:=
\left\langle
	\mcF'_{\beta, \alpha}[1], 
	\mcT'_{\beta, \alpha}
	\right\rangle 
\subset D^b(X).  
\]

We also define a central charge function 
$Z_{\beta, \alpha}^{a, b} \colon \Lambda \to \bC$ as 
\[
Z^{a, b}_{\alpha, \beta}:=
-\ch_3^{\beta}+bH\ch_2^\beta+aH^2\ch_1^\beta
+i\left(
	H\ch_2^\beta-\frac{1}{2}\alpha^2H^3\ch_0^\beta 
	\right) 
\]
for real numbers $a, b \in \bR$.  

Finally, for a real number $\gamma >0$, 
define $U_\gamma$ to be a set of vectors 
$(\alpha, \beta, a, b) \in \bR^4$ satisfying 
\begin{equation} \label{eq:Ugamma}
\alpha> 0, \quad 
\alpha^2+\left(\beta-\lfloor\beta \rfloor-\frac{1}{2} \right)^2 > \frac{1}{4}, 
\quad a > \frac{1}{6}\alpha^2+\frac{1}{2}|b|\alpha+\gamma. 
\end{equation}

\begin{thm}[cf. {\cite[Proposition 8.10]{bms16}}] \label{thm:familystab}
Let $X$ be a triple (resp. double) cover CY3, 
and put $\gamma:=2/9$ (resp. 1/3). 
Then there exists an injective continuous map 
\[
U_\gamma \hookrightarrow \Stab(X), \quad 
(\alpha, \beta, a, b) \mapsto 
\left(Z_{\beta, \alpha}^{a, b}, \mcA_{\beta, \alpha} \right).
\]

Furthermore, the orbit 
$\widetilde{\GL^+}(2, \bR) \cdot U_\gamma$ 
forms an open subset in the space $\Stab(X)$ 
of stability conditions. 
\end{thm}

We divide the proof of the above theorem 
into several steps. 
The arguments below are 
essentially the same as that in \cite[Section 8]{bms16}.

\begin{prop}[cf. {\cite[Theorem 8.6]{bms16}}] \label{prop:axiom12}
For every element $(\alpha, \beta, a, b) \in U_\gamma$ 
with $\alpha, \beta \in \mathbb{Q}$, 
the pair 
$\left(Z_{\beta, \alpha}^{a, b}, \mcA_{\beta, \alpha} \right)$ 
satisfies axioms (1) and (2) 
in Definition \ref{defin:stab}. 
\end{prop}

\begin{proof}
First we check the axiom (1) 
in Definition \ref{defin:stab} 
for the pair 
$\left(Z_{\beta, \alpha}^{a, b}, \mcA_{\beta, \alpha} \right)$. 
As in the proof of \cite[Theorem 8.6]{bms16}, 
it is enough to show the inequality 
$Z_{\beta, \alpha}^{a, b}(F[1]) <0$ 
for every $\nu'_{\beta, \alpha}$-semistable object 
$F$ with $\nu'_{\beta, \alpha}(F)=0$. 
By the inequality 
$\alpha^2+\left(\beta-\lfloor\beta \rfloor-\frac{1}{2} \right)^2 
> \frac{1}{4}$ 
in (\ref{eq:Ugamma}), 
we can apply Theorem \ref{thm:main2intro} to 
the object $F$. 
Noting the equation (\ref{eq:compare}), 
we get 
\begin{equation} \label{eq:ch3ineq}
\ch_3^\beta(F) \leq 
\left(\gamma+\frac{1}{6}\alpha^2 \right)H^2\ch_1^\beta(F). 
\end{equation}
Furthermore, together with the assumption 
$H\ch_2^\beta(F)=\frac{1}{2}\alpha^2H^3\ch_0^\beta$, 
the classical BG inequality (cf. \cite[Theorem 3.5]{bms16}) 
$\overline{\Delta}_H(F) \geq 0$ gives the inequality 
\begin{equation} \label{eq:ch2ineq}
\left(H\ch_2^\beta(F) \right)^2 
\leq \frac{1}{4}\alpha^2\left(H^2\ch_1^\beta(F) \right)^2. 
\end{equation}

By the inequalities 
(\ref{eq:ch3ineq}) and (\ref{eq:ch2ineq}), 
we obtain 
\begin{align*}
Z_{\beta, \alpha}^{a, b}(F[1])
&=\Re Z_{\beta, \alpha}^{a, b}(F[1]) \\
&\leq \left(\gamma+\frac{1}{6}\alpha^2 \right)H^2\ch_1^\beta(F) 
	+\frac{1}{2}|b|\alpha H^2\ch_1^\beta(F) 
	-aH^2\ch_1^\beta(F) 
<0. 	
\end{align*}

Now the axiom (2) is also satisfied 
since we assume $\alpha, \beta \in \mathbb{Q}$ 
(see the proof of \cite[Theorem 8.6]{bms16} for the detail). 
\end{proof}

Next we discuss about the support property. 
Let us put 
\begin{align*}
\overline{\nabla}^{\alpha, \beta, \gamma}_H(E)
&:=3\gamma\alpha^2\left(H^3\ch_0^\beta(E) \right)^2 
+2\left(H\ch^\beta_2(E) \right)\left(
	2H\ch^\beta_2(E)-3\gamma H^3 \ch^\beta_0(E)
	\right) \\
&\quad -6\left(H^2\ch^\beta_1(E) \right) \left(
	\ch^\beta_3(E)-\gamma H^2\ch_1^\beta(E) 
	\right). 
\end{align*}
for an object $E \in D^b(X)$. 

\begin{prop}[cf. {\cite[Lemmas 8.5, 8.8]{bms16}}] \label{prop:support}
Fix an element $(\alpha, \beta, a, b) \in U_\gamma$. 
Then there exists an interval 
$I^{a, b}_{\alpha, \gamma} \subset \bR$ such that 
for every $K \in I^{a, b}_{\alpha, \gamma}$, 
the quadratic form 
$Q^{\alpha, \beta, \gamma}_K:=
K\overline{\Delta}_H
+\overline{\nabla}^{\alpha, \beta, \gamma}_H$ 
is negative definite on the kernel of 
the central charge function 
$Z_{\beta, \alpha}^{a, b}$. 

Furthermore, if we assume 
$(\alpha, \beta) \in \mathbb{Q}$, 
then every $Z_{\beta, \alpha}^{a, b}$-semistable object $E$ 
satisfies the inequality 
$Q^{\alpha, \beta, \gamma}_K(E) \geq 0$. 
\end{prop}
\begin{proof}
Let us prove the first assertion. 
The vectors 
$\left(1, 0, \frac{1}{2}\alpha^2, \frac{1}{2}b\alpha \right)$ 
and $(0, 1, 0, a)$ forms the basis 
the kernel of 
$Z_{\beta, \alpha}^{a, b}$. 
With respect to this basis, 
the quadratic form 
$K\overline{\Delta}_H
+\overline{\nabla}^{\alpha, \beta, \gamma}_H$ 
is represented by the matrix 
\begin{equation} \label{eq:matrix}
\left(
\begin{array}{cc}
-\alpha^2K+\alpha^4 & -\frac{3}{2}b\alpha^2 \\
-\frac{3}{2}b\alpha^2 & K-6(a-\gamma)
\end{array}
\right). 
\end{equation}

When $b=0$, the matrix (\ref{eq:matrix}) 
is negative definite if and only if 
$K \in \left(\alpha^2, 6(a-\gamma) \right)
=:I^{a, b=0}_{\alpha, \gamma}$. 
This interval is non-empty by (\ref{eq:Ugamma}). 

When $b \neq 0$, we also need to require 
the determinant of the matrix (\ref{eq:matrix}) 
to be positive, i.e., 
\begin{equation} \label{eq:detmatrix}
\alpha^2\left(
	-K^2+\left(6(a-\gamma)+\alpha^2 \right)K
	-6(a-\gamma)\alpha^2
-\frac{9}{4}b^2\alpha^2
\right)>0. 
\end{equation}

The solution space $K \in I_{\alpha, \gamma}^{a, b}$ 
of the inequality (\ref{eq:detmatrix}) 
forms a non-empty open interval 
since we have, by (\ref{eq:Ugamma}), 
\[
a-\gamma > \frac{1}{6}\alpha^2+\frac{1}{2}|b|\alpha. 
\]

We can prove the second assertion 
as in \cite[Lemma 8.8]{bms16}, 
using the BG type inequality 
obtained in Theorem \ref{thm:main2intro}. 
\end{proof}

Finally we are able to prove Theorem \ref{thm:familystab}. 

\begin{proof}[Proof of Theorem \ref{thm:familystab}]
By Propositions 
\ref{prop:axiom12} and \ref{prop:support}, 
the pair 
$\left(Z_{\beta, \alpha}^{a, b}, 
\mcA^{\beta, \alpha} \right)$ 
is a stability condition on $D^b(X)$ 
for every element 
$(\alpha, \beta, a, b) \in U_\gamma$ 
with $\alpha, \beta \in \mathbb{Q}$. 
We can deform them to the real parameters 
$(\alpha, \beta)$ by the support property 
in Proposition \ref{prop:support}. 
See \cite[Proposition 8.10]{bms16} 
for the precise proof. 
\end{proof}

\end{document}